\pdfoutput=1
\RequirePackage{fix-cm}
\documentclass[smallextended]{svjour3}       
\smartqed  
 \usepackage{mathptmx}      
\usepackage{graphicx,color}
\usepackage{overpic}
\usepackage{amsmath,amssymb}
\newcommand{\onefig}[1]{
	\begin{figure} 
		\centering
		#1
	\end{figure}
}
\newcommand{\henbi}[2]{\frac{\partial #1}{\partial #2}}
\newcommand{\R}{\mathbb R}
\newcommand{\set}[1]{\left\{#1\right\}}
\newcommand{\norm}[1]{\left\lVert#1\right\rVert}

\newcommand{\eps}{\varepsilon}
\newcommand{\abs}[1]{\left\lvert#1\right\rvert}
\newcommand{\mynorm}[1]{\norm{#1}}

\newcommand{\sobw}[2]{
	\ifx #10
	L^{#2}
	\else 
	\if #22
	H^{#1}
	\else
	W^{#1,#2}
	\fi
	\fi
}
\newcommand{\normsob}[3]{
	\if #32
	\norm{#1}_{#2}
	\else
	\norm{#1}_{#2,#3}
	\fi
}
\newcommand{\snormsob}[3]{
	\if #32
	\abs{#1}_{#2}
	\else
	\abs{#1}_{#2,#3}
	\fi
}
\newcommand{\normi}[1]{\lVert#1\rVert}
\newcommand{\absi}[1]{\lvert#1\rvert}
\newcommand{\normsobi}[3]{
	\if #32
	\normi{#1}_{#2}
	\else
	\normi{#1}_{#2,#3}
	\fi
}
\newcommand{\snormsobi}[3]{
	\if #32
	\absi{#1}_{#2}
	\else
	\absi{#1}_{#2,#3}
	\fi
}
\newcommand{\schemezero}{Scheme LG}
\newcommand{\jurai}{Scheme LG$^\prime$}
\newcommand{\zettai}{Scheme LG-LLV}
\spnewtheorem*{Thschemezero}{Scheme LG}{\bf}{}
\spnewtheorem*{Thschemejurai}{Scheme LG$^\prime$}{\bf}{}
\spnewtheorem*{Thschemezettai}{Scheme LG-LLV}{\bf}{}
\spnewtheorem{hypo}{Hypothesis}{\bf}{}
\spnewtheorem{prob}{Problem}{\bf}{}
\spnewtheorem{exam}{Example}{\bf}{}

\newcommand{\hypoMiniNum}{1$^\prime$}
\newcommand{\lemmMiniNum}{10$^\prime$}
\spnewtheorem*{hypoMini}{Hypothesis \hypoMiniNum}{\bf}{} 
\spnewtheorem*{lemmMini}{Lemma \lemmMiniNum}{\bf}{}
\newcommand{\addedsix}[1]{\textcolor{red}{#1}}
\newcommand{\addedseven}[1]{\textcolor{blue}{#1}}
\newcommand{\donemodify}[1]{#1}

\newcommand{\figwidthbasic}{40mm}
\newcommand{\apoin}{\alpha_1}
\newcommand{\aintone}{\alpha_{20}}
\newcommand{\ainteh}{\alpha_{21}} 
\newcommand{\aintfour}{\alpha_{22}}
\newcommand{\ainvzero}{\alpha_{30}}
\newcommand{\ainvone}{\alpha_{31}}
\newcommand{\astokes}{\alpha_{50}}
\newcommand{\adr}{\alpha_{6}}
\newcommand{\aon}{\alpha_{51}}
\newcommand{\htheo}{h_0}
\newcommand{\htwo}{\htheo}

\newcommand{\czero}{c_0}
\newcommand{\done}{1/4}
\newcommand{\cprop}{c_*}
\newcommand{\beinfsup}{\alpha_4}
\newcommand{\be}{\beta}
\newcommand{\beone}{\beta_{21}^*}
\newcommand{\betwo}{\beta_{22}^*}
\newcommand{\ctheoone}{c_1}
\newcommand{\ctheotwo}{c_2}
\newcommand{\ctheothree}{c_3}
\newcommand{\cuint}{\aintone}

\newcommand{\meas}[1]{\operatorname{meas}(#1)}
\newcommand{\triseq}{\{\mathcal T_h\}_{h\downarrow 0}}
\newcommand{\lag}{\Pi_h}
\newcommand{\lagk}[1]{\lag^{(#1)}}
\newcommand{\schemeint}{\lagk 1}
\newcommand{\pro}[1]{\widehat{#1}_h}
\newcommand{\bwd}{\overline D_{\Delta t}}
\newcommand{\jac}[2]{\det\left(\henbi{#1}{#2}\right)}

\newcommand{\pk}[1]{\mathrm{P}_{#1}}
\newcommand{\lih}{\ell^\infty(H^1_0)}
\newcommand{\ltl}{\ell^2(L^2)}
\newcommand{\lil}{\ell^\infty(L^2)}
\newcommand{\elih}{E_{\ell^\infty(H^1_0)}}
\newcommand{\eltl}{E_{\ell^2(L^2)}}
\newcommand{\elil}{E_{\ell^\infty(L^2)}}
\newcommand{\mini}{\pk1\mathrm{+}/\pk1}
\newcommand{\mixn}[1]{\lVert #1 \rVert}
\newcommand{\mone}{\Large$\bullet$\normalsize}
\newcommand{\mtwo}{$\blacksquare$}
\newcommand{\mthree}{$\blacktriangle$}
\newcommand{\mfour}{$\blacktriangledown$}
\newcommand{\mfive}{\Large$\circ$\normalsize}
\newcommand{\msix}{$\square$}
\newcommand{\mseven}{$\vartriangle$}
\newcommand{\meight}{$\triangledown$}
\newcommand{\columnheight}{3mm}
\allowdisplaybreaks
\journalname{myjournal}
\begin{document}
\title{
A Lagrange--Galerkin scheme with a locally linearized velocity for the Navier--Stokes equations
}
\titlerunning{A Lagrange--Galerkin scheme with a locally linearized velocity}        
\author{ Masahisa Tabata         \and
        Shinya Uchiumi
}
\institute{M. Tabata \at
	Department of Mathematics, Waseda University,
	3-4-1, Ohkubo, Shinjuku, Tokyo 169-8555, Japan \\
	\email{tabata@waseda.jp}           
    \and
	S. Uchiumi 
	\at
	Research Fellow of Japan Society for the Promotion of Science \\
	Graduate School of Fundamental Science and Engineering, Waseda University,
	3-4-1, Ohkubo, Shinjuku, Tokyo 169-8555, Japan \\
	\email{su48@fuji.waseda.jp} 
}
\date{\today}
\maketitle
\begin{abstract}
We present a Lagrange--Galerkin scheme free from numerical quadrature for the Navier--Stokes equations.
Our idea is to use a locally linearized velocity and the backward Euler method in finding the position of fluid particle at the previous time step.    
Since the scheme can be implemented exactly as it is, the theoretical stability and convergence results are assured.
While the conventional Lagrange--Galerkin schemes may encounter the instability caused by numerical quadrature errors, the present scheme is genuinely stable.
For the $\pk 2/\pk 1$- and $\mini$-finite elements 
optimal error estimates are proved in $\ell^\infty(H^1)\times \ell^2(L^2)$ norm for the velocity and pressure.
We present some numerical results, which reflect these estimates and also show the genuine stability of the scheme.    
\end{abstract}
\section{Introduction}
The purpose of this paper is to present a Lagrange--Galerkin scheme free from numerical quadrature for the Navier--Stokes equations and to prove the convergence.  
The Lagrange--Galerkin method, 
which is also called characteristics finite element method or Galerkin-characteristics method,
is a powerful numerical method for flow problems, having such advantages that it is robust for convection-dominated problems and that the resultant matrix to be solved is symmetric.  
It has, however, a drawback that it may lose the stability when numerical quadrature is employed to integrate composite function terms that characterize the method.     
Our scheme presented here overcomes this drawback.    

Lagrange--Galerkin schemes for the Navier--Stokes equations have been developed in  \cite{AchdouGuermond2000,BermejoSastreSaavedra2012,BoukirEtal,NotsuTabata2009,2015_Notsu-T2,Pironneau1982,Priestley1994,Suli}; see also bibliography therein.
After convergence analysis was done successfully by Pironneau \cite{Pironneau1982} in a suboptimal rate, the optimal convergence result was obtained by S\"uli \cite{Suli}. 
Optimal convergence results by Lagrange--Galerkin schemes were extended to the multi-step method by Boukir et al. \cite{BoukirEtal}, 
to the projection method by Achdou--Guermond \cite{AchdouGuermond2000}
and 
to the pressure-stabilized method by Notsu-Tabata \cite{2015_Notsu-T2}.  
All these results of the stability and convergence are proved under the condition that the integration of the composite function terms is computed exactly.  
Since it is difficult to perform the exact integration in real problems, numerical quadrature is usually employed.    
It is, however, reported that instability may occur caused by numerical quadrature error for convection-diffusion problems in \cite{MPS,Priestley1994,Tabata2007,TSTeng} .  
We observe such instability occurs for the Navier-Stokes equations by numerical examples in this paper.  

Several methods have been studied to avoid the instability in \cite{BermejoSastreSaavedra2012,MPS,PironneauTabata2010,Priestley1994,TSTeng}. 
The map of a fluid particle from the present position to the position a time increment  $\Delta t$ before (the position is often called foot along the trajectory) is simplified.  
To find the foot of a particle is nothing but to solve a system of ordinary differential equations (ODEs).  
Morton et al. \cite{MPS} solved the ODEs only at the centroids of the elements, and Priestley \cite{Priestley1994} did only at the vertices of the elements.  
The map of the other points is approximated by linear interpolation of those values.    
It becomes possible to perform the exact integration of the composite function terms with the simplified map.  
Bermejo et al. \cite{BermejoSastreSaavedra2012} used the same simplified map as  \cite{Priestley1994} to employ a numerical quadrature of high accuracy to the composite function terms for the Navier-Stokes equations.   
Tanaka et al. \cite{TSTeng} and Tabata--Uchiumi \cite{TabataUchiumi1} approximated the map by a locally linearized velocity and the backward Euler approximation to solve the ODEs for convection-diffusion problems.  
The approximate map makes possible the exact integration of the composite function terms.  

In this paper we prove the convergence
of a Lagrange--Galerkin scheme with the same approximate map as \cite{TabataUchiumi1,TSTeng} in the $\pk2/\pk1$- or $\mini$-element for the Navier--Stokes equations.    
Since we neither solve the ODEs nor use numerical quadrature, our scheme can be precisely implemented to realize the theoretical results.  
It is, therefore, a genuinely stable Lagrange--Galerkin scheme. 
Our convergence results are best possible for the velocity and pressure in $\ell^\infty(H^1)\times \ell^2(L^2)$-norm for both elements as well as for the velocity in $\ell^\infty(L^2)$-norm in the $\mini$-finite element.

The contents of this paper are as follows.  
In the next section we describe the Navier--Stokes problem and some preparation.
In Section \ref{sec:scheme}, after recalling the conventional Lagrange--Galerkin scheme, we present our Lagrange--Galerkin scheme with a locally linearized velocity. 
In Section \ref{sec:mainResults} we show convergence results, which are proved in Section \ref{sec:proof}.
In Section \ref{sec:numex} we show some numerical results, which reflect the theoretical convergence orders and the robustness of the scheme for high Reynolds number problems.  
In Section \ref{sec:conclusion} we give the conclusions.
\section{Preliminaries}
We state the problem and prepare the notation used throughout this paper.

Let $\Omega$ be a polygonal or polyhedral domain of $\R^d ~ (d=2,3)$ and
$T>0$ a time.
We use the Sobolev spaces $L^p(\Omega)$ with the norm $\normsob{\cdot}{0}{p}$, 
$\sobw{s}{p}(\Omega)$ and $W^{s,p}_0(\Omega)$ 
with the norm $\normsob{\cdot}{s}{p}$ and the semi-norm $\snormsob{\cdot}{s}{p}$ 
for $p=1, 2,\infty$ 
and a positive integer $s$.
When $p=2$, we write $H^s(\Omega) = W^{s,2}(\Omega)$ simply and drop the subscript $2$ in the corresponding norms.
For the vector-valued function $w\in W^{1,\infty}(\Omega)^d$ we define the semi-norm $\snormsobi{w}{1}{\infty}$ by
\begin{equation*}
	\biggl\lVert\biggl\{\sum_{i,j=1}^d \left( \henbi{w_i}{x_j} \right)^2 \biggr\}^{1/2}\biggr\rVert_{0,\infty}.
\end{equation*}
$L^2_0(\Omega)$ is the subspace of $L^2(\Omega)$ with the zero mean.
The parenthesis $(\cdot , \cdot)$ shows the $L^2(\Omega)^i$-inner product for $i=1, d$ or $d\times d$.
For $w\in L^2(\Omega)^d$, $\norm{w}_{-1}$ stands for the dual norm 
\begin{equation*}
	\sup_{v\in H^1_0(\Omega)^d\setminus \{0\}} \frac{(w,v)}{\norm{v}_{1}}.
\end{equation*}
For a Sobolev space $X(\Omega)$ 
we use the abbreviations 
$H^m(X)=H^m(0,T;X(\Omega))$ and $C(X)=C([0,T];X(\Omega))$.
\begin{equation*}
\begin{split}
Z^m(t_1,t_2) &\equiv \set{f\in H^j(t_1, t_2;H^{m-j}(\Omega)^d); j=0,\dots,m,\mynorm{f}_{Z^m(t_1,t_2)}<\infty},\\
\mynorm{f}_{Z^m(t_1,t_2)} &\equiv 
\biggl\{ \sum_{j=0}^m \mynorm{f}_{H^j(t_1,t_2;H^{m-j}(\Omega)^d)}^2 \biggr\}^{1/2}
\end{split}
\end{equation*}
and denote $Z^m(0,T)$ by $Z^m$.

We consider the Navier--Stokes equations:
find $(u,p):\Omega \times (0,T) \to \R^d \times \R$
such that
\begin{equation}\label{NS}
	\begin{split}
	\frac{Du}{Dt}
	- \nu \Delta u + \nabla p = f & \quad \text{in~} \Omega \times (0,T),\\
	\nabla \cdot u = 0 & \quad \text{in~} \Omega \times (0,T),\\
	u = 0 & \quad \text{on~} \partial \Omega \times (0,T),\\
	u = u^0 & \quad \text{in~} \Omega \text{~at~} t=0,
	\end{split}
\end{equation}
where 
$\partial \Omega$ is the boundary of $\Omega$,
$\frac{Du}{Dt} \equiv \henbi{u}{t}+(u \cdot \nabla) u$ is the material derivative and
$\nu>0$ is a viscosity. 
Functions $f\in C(L^2)$ 
and $u^0:\Omega \to \R^d$ are given.

We define the bilinear forms $a$ on $H^1_0(\Omega)^d\times H^1_0(\Omega)^d$ and $b$ on $H^1_0(\Omega)^d\times L^2_0(\Omega)$ by 
\begin{equation*}
a(u,v) \equiv \nu(\nabla u,\nabla v), \quad
b(v,q) \equiv -(\nabla \cdot v,q).
\end{equation*}
Then, we can write the weak form of (\ref{NS}) as follows:  
find $(u,p) : (0,T)\to H^1_0(\Omega)^d \times L^2_0(\Omega)$ such that for $t\in (0,T)$, 
\begin{subequations}\label{eq:NSweak}
\begin{align}
&&	\left(\frac{Du}{Dt}(t),v\right) + a(u(t),v)+b(v,p(t))&=(f(t),v), & \quad \forall v\in H^1_0(\Omega)^d, \\
&&	b(u(t),q)&=0, & \quad \forall q\in L^2_0(\Omega),
\end{align}
\end{subequations}
with $u(0)=u^0$.

Let $u$ be smooth.
The characteristic curve $X(t; x,s)$ 
is defined by the solution of the system of the ordinary differential equations, 
\begin{subequations}\label{odes}
\begin{align}
	\frac{dX}{dt}(t; x,s)&=u(X(t;x,s),t), \quad t<s,\\
	X(s;x,s)&=x.
\end{align}
\end{subequations}
Then, we can write the material derivative term $(\henbi{}{t}+u\cdot \nabla)u$ as follows: 
\begin{equation*}
\left(\henbi{u}{t} + (u \cdot \nabla) u \right)(X(t),t)=\frac{d}{dt} u(X(t),t).
\end{equation*}

Let $\Delta t>0$ be a time increment.  
For $w:\Omega \to \R^d$ we define the mapping $X_1(w):\Omega \to \R^d$ by
\begin{equation}\label{eq:x1def}
(X_1(w))(x) \equiv x - w(x)\Delta t.
\end{equation}
\begin{remark}
The image of $x$ by $X_1(u(\cdot,t))$ is nothing but the approximate value of $X(t-\Delta t;x,t)$ obtained by solving \eqref{odes} by the backward Euler method.
\end{remark}

Let $N_T \equiv \lfloor T/\Delta t \rfloor$, 
$t^n \equiv n\Delta t$ and 
$\psi^n \equiv \psi(\cdot,t^n)$ for a function $\psi$ defined in $\Omega \times (0,T)$.
For a set of functions $\psi=\set{\psi^n}_{n=0}^{N_T}$
and a Sobolev space $X(\Omega)$, 
two norms 
$\mynorm{\cdot}_{\ell^\infty(X)}$ and $\mynorm{\cdot}_{\ell^2(n_1, n_2;X)}$
are defined by
\begin{equation*}
\begin{split}
	\norm{\psi}_{\ell^\infty(X)} &\equiv 
	\max \left\{ \norm{\psi^n}_{  X(\Omega)  };n=0,\dots ,N_T \right\}, \\
	\norm{\psi}_{\ell^2(n_1,n_2;X)} &\equiv 
	\biggl(
	\Delta t \sum_{n=n_1}^{n_2} \norm{\psi^n}_{X(\Omega)}^2  
	\biggr)^{1/2}, 
\end{split}
\end{equation*}
and $\mynorm{\psi}_{\ell^2(1,N_T;X)}$ is denoted by $\mynorm{\psi}_{\ell^2(X)}$. 
The backward difference operator $\bwd$ is defined by
\begin{equation*}
	\bwd \psi^n \equiv \frac{\psi^n-\psi^{n-1}}{\Delta t}.
\end{equation*}
%
%

Let $\mathcal T_h$ be a triangulation of $\bar \Omega$ and 
$h \equiv \max_{K\in \mathcal T_h} \operatorname{diam}(K)$ the maximum element size.
Throughout this paper we consider a regular family of triangulations $\{ \mathcal{T}_h  \}_{h\downarrow 0}$.  
Let $V_h \times Q_h \subset H_0^1(\Omega)^d \times L_0^2(\Omega)$ be the $\pk2$/$\pk1$- or $\mini$-finite element space, which is called Hood-Taylor element or MINI element \cite{GiraultRaviart,MINI}.
Let 
\[
  \lagk{1} : C(\bar \Omega)^d \cap H^1_0(\Omega)^d \to V_h
\] 
be the Lagrange interpolation operator to the 
$\pk 1$-finite element space.
Let $(\pro{w},\pro{r})\equiv \Pi_h^{S}(w,r) \in V_h \times Q_h$ be the Stokes projection of $(w,r) \in H^1_0(\Omega)^d \times L^2_0(\Omega)$ defined by 
\begin{subequations}\label{eq:stokesproj}
\begin{align}
&& a(\pro w, v_h)+b(v_h,\pro r) &= a(w,v_h)+b(v_h,r), \quad &\forall v_h\in V_h, \\
&& b(\pro w,q_h)&=b(w,q_h),  &\forall q_h \in Q_h.
\end{align}
\end{subequations}
We denote by $(\Pi_h^{S}(w,r))_1$ the first component $\pro{w}$ of $\Pi_h^{S}(w,r)$.  

The symbol $\circ$ stands for the composition of functions, e.g., $(g\circ f)(x) \equiv g(f(x))$.

\section{A Lagrange--Galerkin scheme with a locally linearized velocity}\label{sec:scheme} 
The conventional Lagrange--Galerkin scheme, which we call \schemezero, is described as follows.
\begin{Thschemezero}
Let $u_h^0 = (\Pi^{S}_h (u^0,0))_1$. 
Find $\set{(u_h^n, p_h^n)}_{n=1}^{N_T} \subset V_h \times Q_h$ such that 
\begin{subequations}\label{eq:scheme0}
\begin{align*}
&&	\left( \frac{u_h^{n} - u_h^{n-1}\circ  X_1(u_h^{n-1}) 	}{\Delta t}, v_h\right) + a(u_h^{n}, v_h)
	+ b(v_h, p_h^{n}) 	&= (f^{n},v_h), \quad  	&\forall v_h \in V_h, \\
&&	b(u_h^{n}, q_h) &= 0, \quad 	&\forall q_h \in Q_h,
\end{align*}
\end{subequations}
for $n=1,\dots, N_T$.
\end{Thschemezero}
\begin{remark}
S\"uli \cite{Suli} used the exact solution $X_h^{n-1}$ of the system of ordinary differential equations, 
\begin{subequations}\label{eq:xhSuli} 
\begin{align}
	\frac{d X_h^{n-1}}{dt} (t;x,t^n) &= u_h^{n-1}(X_h^{n-1}(t;x,t^n),t), \quad t^{n-1}<t<t^n,\\
	X_h^{n-1}(t^n;x,t^n) &= x
\end{align}
\end{subequations}
instead of $X_1(u_h^{n-1})$.
\end{remark}
By a similar way to \cite{Suli} combined with \cite{BoukirEtal}, error estimates
\begin{subequations}
\begin{align}
	\mynorm{u_h-u}_{\ell^\infty(H^1)}, \mynorm{p_h-p}_{\ell^2(L^2)} &\leq c(h^k+\Delta t),\\
	\mynorm{u_h-u}_{\ell^\infty(L^2)} &\leq c(h^{k+1}+\Delta t), \label{eq:estimeteScheme0ON}
\end{align}
\end{subequations}
can be proved, where $k=2$ for $\pk2/\pk1$-element and $k=1$ for $\mini$-element.
In the estimate above,
the composite function term
	$( u_h^{n-1} \circ X_1(u_h^{n-1}), v_h )$
is assumed to be exactly integrated.

Although the function $u_h^{n-1}$ is a polynomial on each element $K$, the composite function $u_h^{n-1} \circ X_1(u_h^{n-1})$ is not a polynomial on $K$ in general since the image $X_1(u_h^{n-1})$ of an element $K$ may spread over plural elements.
Hence, it is hard to calculate the composite function term $( u_h^{n-1} \circ X_1(u_h^{n-1}), v_h )$ exactly.
In practice, the following numerical quadrature has been used.
Let 
$g:K \to \R$ be a continuous function. 
A numerical quadrature $I_h[g;K]$ of $\int_K g ~ dx$ is  defined by 
\begin{equation*}\label{eq:numericalQuadrature}
I_h[g;K] \equiv \meas{K} \sum_{i=1}^{N_q} w_i g(a_i),
\end{equation*}
where $N_q$ is the number of quadrature points and $(w_i, a_i) \in \R \times K$ is a pair of the weight and the point for $i=1, \dots, N_q$.
We call the practical scheme using numerical quadrature \jurai.
\begin{Thschemejurai}
Let $u_h^0 = (\Pi^{S}_h (u^0,0))_1$. 
Find $\set{(u_h^n, p_h^n)}_{n=1}^{N_T} \subset V_h \times Q_h$ such that 
\begin{align*}
	\frac{1}{\Delta t} (u_h^{n}, v_h) 
	- \frac{1}{\Delta t} 
	\sum_{K \in \mathcal T_h} I_h[(u_h^{n-1}\circ X_1(u_h^{n-1}))\cdot v_h;K] 
	& &\\
	+ a(u_h^{n}, v_h)
	+ b(v_h, p_h^{n}) 
	&= (f^{n},v_h), &\forall v_h \in V_h, \\
	b(u_h^{n}, q_h) &= 0, &\forall q_h \in Q_h, 
\end{align*}
for $n=1,\dots, N_T$.
\end{Thschemejurai}
For convection-diffusion equations it has been reported that numerical quadrature causes the instability  \cite{MPS,Priestley1994,Tabata2007,TabataFujima2006,TabataUchiumi1,TSTeng}.  
For the Navier-Stokes equations we present numerical results showing the instability of {\jurai} in Section \ref{sec:numex}. 

We now present our Lagrange-Galerkin scheme with a locally linearized velocity.  
It is free from quadrature and exactly computable.
We call it \zettai.  
\begin{Thschemezettai}
Let $u_h^0 = (\Pi^{S}_h (u^0,0))_1$.
Find $\set{(u_h^n, p_h^n)}_{n=1}^{N_T} \subset V_h \times Q_h$ such that 
\begin{subequations}\label{eq:scheme2}
\begin{align}
\!\!\!\!\!\!
\biggl( \frac{u_h^{n} - u_h^{n-1}\circ  X_1(\schemeint u_h^{n-1}) }{\Delta t}, v_h\biggr)
	+ a(u_h^{n}, v_h) 	+ b(v_h, p_h^{n}) 	&= (f^{n},v_h),  &\forall v_h \in V_h, \\
\!\!\!\!\!\!
b(u_h^{n}, q_h) &= 0, &\forall q_h \in Q_h,
\end{align}
\end{subequations}
for $n=1\dots, N_T$.
\end{Thschemezettai}

In the above scheme the locally linearized velocity $\lagk 1 u_h^{n-1}$ is used in place of the original velocity $u_h^{n-1}$.  
The error caused by the introduction of the approximate velocity $\lagk 1 u_h^{n-1}$ is evaluated properly in Theorems \ref{theo:mainConvergence} and \ref{theo:mini} in the next section.  
The following proposition assures that
the integration $(u_h^{n-1}\circ X_1(\lagk 1 u_h^{n-1}),v_h)$ can be calculated exactly. 
\begin{proposition}\label{exactcomp}
  Let $u_h$, $v_h\in V_h$ and $w\in W^{1,\infty}_0(\Omega)^d$. 
  Suppose 
  $ \aintone \Delta t \snormsobi{w}{1}{\infty} <1$,
  where  $\aintone$ is the constant defined in \eqref{eq:aintone} below.
  Then, $\int_{\Omega} (u_h \circ X_{1}(\lagk 1 w)) \cdot v_h \ dx$ is exactly computable.
\end{proposition}

{\it Outline of the proof.}~
When $u_h$ and $v_h$ are scalar functions, the result on the exact computability has been proved in \cite{TSTeng} and \cite[Proposition 1]{TabataUchiumi1}.   
Here, we do not repeat the proof but show only the outline.   
It is necessary that the inclusion $(X_1(\lagk{1} w))(\Omega)\subset \Omega$ holds to execute the integration of $u_h \circ X_{1}(\lagk 1 w)) \cdot v_h$ over $\Omega$.  
The condition $ \aintone \Delta t \snormsobi{w}{1}{\infty} <1$ is sufficient for it by virtue of Lemma 
\ref{lemm:jacobiest}-(i) and \eqref{eq:aintone} below.
The mapping $X_1(\schemeint w)$ is linear on each element.  
When a mapping $F$ is linear, we have the following general result for any two elements $K_0$ and $K_1$ and any polynomial $\phi_h$ of any order $k$ defined on $K_1$.     
Proposition \ref{exactcomp} is proved by applying the following lemma, whose proof is easy, cf. \cite[Lemma 1]{TabataUchiumi1}.
\begin{lemma}\label{lemm:polygon}
 Let $K_0, K_1\in \mathcal T_h$ and $F:K_0\to \R^d$ be linear and one-to-one.
 Let 
 $
	E_1 \equiv K_0 \cap F^{-1} (K_1)
 $
 and $\meas{E_1}>0$. 
 Then, the following hold.
 \begin{enumerate}
 \item $E_1$ is a polygon ($d=2$) or a polyhedron ($d=3$).
 \item $\phi_h \circ F_{|E_1}\in \pk k (E_1), \quad \forall \phi_h \in \pk k(K_1)$.
 \end{enumerate}
\end{lemma}
\begin{remark}
In the case of $d=2$, 
Priestley \cite{Priestley1994} approximated
$X_h^{n-1}(t^{n-1};x,t^n)$ in (\ref{eq:xhSuli}) 
by
\begin{equation*}
	\widetilde X_{h}(x) = B_1 \lambda_1(x)+B_2 \lambda_2(x)+B_3 \lambda_3(x), ~ x\in K_0
\end{equation*}
on each $K_0 \in \mathcal T_h$, where $B_i=X_h^{n-1}(t^{n-1};A_i,t^n)$, $\set{A_i}_{i=1}^3$ are vertices of $K_0$ and $\set{\lambda_i}_{i=1}^3$ are the barycentric coordinates of $K_0$ with respect to $\set{A_i}_{i=1}^3$.
Since $\widetilde X_{h}(x)$ is linear in $K_0$, the decomposition
\begin{align*}
	\int_{K_0} (u_h^{n-1} \circ \widetilde X_{h}) \cdot v_h \ dx
	= \sum_{l \in \Lambda(K_0)} \int_{E_l} (u_h^{n-1} \circ \widetilde X_{h}) \cdot v_h \ dx, \\
	\Lambda(K_0) \equiv \set{l; K_0\cap \widetilde X_{h}^{-1}(K_l) \not = \emptyset}, \quad
	E_l \equiv K_0 \cap \widetilde X_{h}^{-1}(K_l) 
\end{align*}
makes the exact integration possible. 
However, $B_i=X_h^{n-1}(t^{n-1};A_i,t^n)$ are the solutions of a system of ordinary differential equations and
it is not easy to solved it exactly in general since $u_h^{n-1}$ is piecewise polynomial. 
In practice, some numerical method, e.g., Runge--Kutta method, is required, which introduces another error.
\end{remark}

\section{Main results}\label{sec:mainResults}
We present the main results of error estimates for \zettai, which are proved in the next section.
We first state the result when the $\pk2/\pk1$-element is employed.
\begin{hypo}\label{hypo:NSreg}
The solution of (\ref{NS}) satisfies 
\begin{equation*}
	u\in 
	Z^2 \cap H^1(H^3), ~
	p\in H^1(H^2).
\end{equation*}
\end{hypo}
\begin{remark}\label{divu00}
Hypothesis \ref{hypo:NSreg} implies $(u,p) \in C(H^3 \times H^2)$, which yields $\nabla \cdot u^0 =0$.
\end{remark}
\begin{hypo}\label{hypo:inverse}
The sequence $\triseq$ satisfies the inverse assumption.
In addition, for each $h$, $\forall K\in \mathcal T_h$ has at least one vertex in $\Omega$.
\end{hypo}
\begin{theorem}\label{theo:mainConvergence}
Let 
$V_h\times Q_h$ be the $\pk2/\pk1$-finite element space.
Suppose Hypotheses \ref{hypo:NSreg} and \ref{hypo:inverse}. 
Then, there exist positive constants 
$\czero$ and $\htheo$ 
such that if $h\in (0,\htheo]$ and $\Delta t \leq \czero h^{d/4}$, 
the solution $(u_h,p_h) \equiv \set{(u_h^n,p_h^n)}_{n=0}^{N_T}$ of {\zettai} exists  
and the estimates
\begin{equation*}\label{eq:mainestimate}
	\mynorm{u_h-u}_{\ell^\infty(H^1)}, \mynorm{p_h-p}_{\ell^2(L^2)} \leq \ctheoone (h^2+\Delta t) 
\end{equation*}
hold, where $\ctheoone$ is a positive constant independent of $h$ and $\Delta t$.
\end{theorem} 
Next, we state the result when the $\mini$-element is employed.
\begin{hypoMini} 
The solution of (\ref{NS}) satisfies 
\begin{equation*}
	u\in 
	Z^2 \cap H^1(H^2), ~
	p\in H^1(H^1).
\end{equation*}
\end{hypoMini}
\begin{remark}\label{divu00_MINI}
Hypothesis 1$^\prime$ implies $(u,p) \in C(H^2 \times H^1)$, which yields $\nabla \cdot u^0 =0$.
\end{remark}
\begin{hypo}\label{hypo:ON}
The Stokes problem is regular, that is, for all $g\in L^2(\Omega)^d$ the solution $(w,r) \in H^1_0(\Omega)^d \times L^2_0(\Omega)$ of the Stokes problem, 
\begin{equation*}\label{eq:stokesStrong}
\begin{split}
	-\nu \Delta w + \nabla r =g, & \quad x \in \Omega,\\
	\nabla \cdot w = 0, & \quad x \in \Omega
\end{split}
\end{equation*}
belongs to $H^2(\Omega)^d\times H^1(\Omega)$ and the estimate
\begin{equation*}
	\mixn{(w,r)}_{H^2 \times H^1} \leq c \normsob{g}{0}{2}.
\end{equation*}
holds, where $c$ is a positive constant independent of $g, w$ and $r$.
\end{hypo}
\begin{remark}
Hypothesis \ref{hypo:ON} holds, for example, if $d=2$ and $\Omega$ is convex \cite{GiraultRaviart}.
\end{remark}
\begin{theorem}\label{theo:mini}
Let 
$V_h\times Q_h$ be the $\mini$-finite element space.
Suppose Hypotheses \hypoMiniNum and \ref{hypo:inverse}.
Then, there exist positive constants 
$\czero$ and $\htheo$ 
such that if $h\in (0,\htheo]$ and $\Delta t \leq \czero h^{d/4}$, 
the solution $(u_h,p_h) \equiv \set{(u_h^n,p_h^n)}_{n=0}^{N_T}$ of {\zettai} exists, 
and the estimates
\begin{equation}\label{eq:miniEst1}
	\mynorm{u_h-u}_{\ell^\infty(H^1)}, \mynorm{p_h-p}_{\ell^2(L^2)} \leq \ctheotwo(h+\Delta t)
\end{equation}
hold, where $\ctheotwo$ is a positive constant independent of $h$ and $\Delta t$.
Moreover, under Hypothesis \ref{hypo:ON}, 
the estimate 
\begin{equation}\label{eq:miniEst2}
	\mynorm{u_h-u}_{\ell^\infty(L^2)} \leq \ctheothree(h^2+\Delta t)
\end{equation}
holds, where $\ctheothree$ is a positive constant independent of $h$ and $\Delta t$.
\end{theorem}
\begin{remark}
The convergence proof is easily extended for any pairs satisfying the inf-sup condition.  
However, the convergence order with respect to the space discretization is bounded by $O(h^2)$ caused by the locally linearized approximation of the velocity.
In fact, in the case of the $\pk2/\pk1$-element the estimate (\ref{eq:estimeteScheme0ON}) with $k=2$ does not hold in \zettai, cf., Example \ref{exam:funcC} in Section \ref{sec:numex}.  
\end{remark}
\section{Proofs of the main theorems}\label{sec:proof}
We prove Theorem \ref{theo:mainConvergence} in Subsections \ref{subsec:p2p1proofbegin}--\ref{subsec:p2p1proofend} and Theorem \ref{theo:mini} in Subsection \ref{subsec:miniproof}.
\subsection{Some lemmas}
We recall some results used in proving the main theorems.
For proofs of Lemmas \ref{lemm:poin}--\ref{lemm:stokespro} we refer to the cited bibliography.
\begin{lemma}[Poincar\'e's inequality \cite{Ciarlet}]\label{lemm:poin}
	There exists a positive constant $\apoin(\Omega)$ such that
	\begin{equation*}
		\mynorm{v}_{0} \leq \apoin \abs{v}_{1}, \quad \forall v\in H^1_0(\Omega)^d.
	\end{equation*}
\end{lemma}
\begin{lemma}[the Lagrange interpolation \cite{Ciarlet}] 
	\label{lemm:intest}
	Suppose $\triseq$ is a regular family of triangulations of $\bar \Omega$.
	Let $X_h$ be the $\pk2$- or $\pk1\mathrm{+}$-finite element space and
	$\lagk{1}$ be the Lagrange interpolation operator to the $\pk 1$-finite element space. 
	Then, it holds that 
	\begin{equation*}
		\normi{\Pi_h^{(1)} v}_{0,\infty} \leq \mynorm{v}_{0,\infty}, \quad \forall v\in C(\bar \Omega)^d,
	\end{equation*}
	and there exist positive 
	constants 
	$\aintone \geq 1$, 
	$\ainteh$
	and $\aintfour$ such that
\begin{subequations}
\begin{align}
&& \absi{\Pi_h^{(1)} v}_{1,\infty} &\leq \aintone \abs{v}_{1,\infty} ,           &\forall v\in W^{1,\infty}(\Omega)^d, \label{eq:aintone}\\
&& \normsobi{\lagk{1} v-v}{s}{2} &\leq \ainteh h^{2-s} \snormsobi{v}{2}{2}, &s=0,1, ~ \forall v\in H^{2}(\Omega)^d,\\
&& \normi{\Pi_h^{(1)}v_h}_0       &\leq \aintfour \mynorm{v_h}_0,            &\forall v_h \in X_h. \label{eq:stabfinitedim}
\end{align}
\end{subequations}
\end{lemma}
\begin{remark}
	The inequality (\ref{eq:stabfinitedim}) holds since $X_h$ is finite-dimensional.
If we replace $\lagk 1$ by the Cl\'ement interpolation operator \cite{Clement}, 	this inequality holds for all $v\in L^2(\Omega)^d$.  
\end{remark}
\begin{lemma}[the inverse inequality \cite{Ciarlet,Suli}]
	\label{lemm:inv}
	Suppose $\triseq$ satisfies the inverse assumption. 
	Let $X_h $ be the $\pk2$- or $\pk1\mathrm{+}$-finite element space.
	Then, there exist positive constants $\ainvzero$ and $\ainvone$ such that
	\begin{equation*}
		\begin{split}
			\mynorm{v_h}_{0,\infty} \leq \ainvzero h^{-d/6} \mynorm{v_h}_1, \quad &\forall v_h \in X_h, \\
			\abs{v_h}_{1,\infty} \leq \ainvone h^{-d/2} \abs{v_h}_1, \quad &\forall v_h \in X_h.
		\end{split}
	\end{equation*}
\end{lemma}
\begin{lemma}[the inf-sup condition \cite{MINI,BercovierPironneau,Verfurth1984}]
	\label{lemm:infsup}
	Suppose Hypothesis \ref{hypo:inverse}. 
	Let $V_h \times Q_h \subset H_0^1(\Omega)^d \times L_0^2(\Omega)$ be the $\pk2/\pk1$- or $\mini$-finite element space.
	Then, there exists a positive constant $\beinfsup$ independent of $h$ such that 
	\begin{equation*}
		\inf_{q_h \in Q_h\setminus\set{0}} \sup_{v_h \in V_h\setminus\set{0}} \frac{b(v_h, q_h)}{\mynorm{v_h}_{1} \norm{q_h}_{0}} \geq \beinfsup.
	\end{equation*}
\end{lemma}
\begin{lemma}[\cite{GiraultRaviart}]
	\label{lemm:stokespro}
	(i) Suppose Hypothesis \ref{hypo:inverse}  
	and that $V_h \times Q_h \subset H_0^1(\Omega)^d \times L_0^2(\Omega)$ is the $\pk2/\pk1$- or $\mini$-finite element space.
	Let $(\pro{w},\pro{r})$ be the Stokes projection of $(w,r)$ defined in (\ref{eq:stokesproj}). 
	Then, there exists a positive constant $\astokes$ independent of $h$ such that
	\begin{equation*}
		\normsob{\pro{w}-w}{1}{2}, \normsob{\pro{r}-r}{0}{2} 
		\leq 
		\astokes h^k \mixn{(w,r)}_{H^{k+1}\times H^{k}},
	\end{equation*}
	where $k=2$ for the $\pk2/\pk1$-element and $k=1$ for the $\mini$-element. \\
		(ii) Moreover, suppose Hypothesis \ref{hypo:ON}.
		Then, there exists a positive constant $\aon$ such that
		\begin{equation*}
			\normsob{\pro{w}-w}{0}{2} \leq 
			\aon h^{k+1} \mixn{(w,r)}_{H^{k+1}\times H^{k}},
		\end{equation*}
		where $k=2$ for the $\pk2/\pk1$-element and $k=1$ for the $\mini$-element.
\end{lemma}
\begin{lemma} 
	\label{lemm:jacobiest}
(i)	Let $w\in W^{1,\infty}_0(\Omega)^d$ and $X_1(w)$ be the mapping defined in (\ref{eq:x1def}).
	Then, under the condition $\Delta t \snormsobi{w}{1}{\infty}<1$, $X_1(w):\Omega \to \Omega$ is bijective.
\\
(ii) Furthermore, under the condition $\Delta t\abs{w}_{1,\infty} \leq \done$,
		the estimate 
		\begin{equation*}
			\frac12 \leq \jac{X_1(w)}{x} \leq \frac32
		\end{equation*}
		holds, where $\det (\partial X_1(w)/\partial x)$ is the Jacobian.
\end{lemma}
\begin{proof}
The former is proved in \cite[Proposition 1]{RuiTabata2002}.
We prove the latter only in the case $d=3$ since the proof in $d=2$ is much easier.
Let $I$ be the $3\times 3$ identity matrix, $A=(a_{ij})$ and $\boldmath{a}_j=(a_{1j}, a_{2j}, a_{3j})^T$,
where $a_{ij}=\Delta t \, \partial w_i/ \partial x_j $ for $i,j=1,2,3$.
The notation $|\cdot|$ stands for the absolute value, or the Euclidean norm in $\R^3$ or $\R^{3\times 3}$.
From the condition
\begin{equation*}
\abs{a}=(\abs{\boldmath a_1}^2 + \abs{\boldmath a_2}^2 + \abs{\boldmath a_3}^2)^{1/2} \leq 1/4,
\end{equation*} 
we obtain 
\begin{equation*}
	\det(A) \leq \abs{\boldmath a_1} \abs{\boldmath a_2} \abs{\boldmath a_3} \leq \left(\frac{1}{4\sqrt{3}} \right)^3.
\end{equation*}
Then, we have
\begin{equation*}
\begin{split}
	&\left\lvert \jac{X_1(w)}{x} - 1 \right\rvert = |\det(I-A) -1| \\
	= & |-(a_{11}+a_{22}+a_{33}) \\
		&+ a_{11}a_{22}+a_{22}a_{33}+a_{33}a_{11} 
		- a_{12}a_{21} - a_{23}a_{32} - a_{31}a_{13} - \det(A)| \\
	\leq & |a_{11}+a_{22}+a_{33}| \\ 
		& +| a_{11}a_{22}+a_{22}a_{33}+a_{33}a_{11} - a_{12}a_{21} - a_{23}a_{32} - a_{31}a_{13}|
		+ |\det(A)|\\
	\leq & \sqrt{3}|a|+|a|^2+|\det(A)|
	\leq 1/2,
\end{split}
\end{equation*}
which implies the result.
\qed
\end{proof}
\begin{lemma}
	\label{lemm:twofunc}
	Let $1\leq q <\infty$, $1\leq p \leq \infty$, $1/p+1/p'=1$ and
	$w_i\in W_0^{1,\infty}(\Omega)^d$, $i=1,2$.
	Under the condition $\Delta t \snormsobi{w_i}{1}{\infty}\leq\done$,
	it holds that,
	for $\psi\in W^{1,qp'}(\Omega)^d$, 
	\begin{equation*}
		\normsob{\psi\circ X_1(w_1) - \psi\circ X_1(w_2)}{0}{q} \leq 
		2^{1/(qp')}
		\Delta t\normsob{w_1-w_2}{0}{pq} 
		\normsob{\nabla \psi}{0}{qp'},
	\end{equation*} 
	where $X_1(\cdot)$ is defined in (\ref{eq:x1def}).
\end{lemma}
Lemma \ref{lemm:twofunc} is a direct consequence of \cite[Lemma 4.5]{AchdouGuermond2000} and Lemma \ref{lemm:jacobiest}-(ii).
\begin{lemma}
	\label{lemm:DouglasRussell}
	Let $w\in W_0^{1,\infty}(\Omega)^d$. 
	Under the condition $\Delta t\snormsobi{w}{1}{\infty}\leq\done$, there exists a positive constant
	$\adr$
	such that, for $\psi \in L^2(\Omega)^d$,
	\begin{equation*}
		\mynorm{\psi-\psi\circ X_1(w)}_{-1} 
		\leq \adr \Delta t \normsob{w}{1}{\infty}  \normsob{\psi}{0}{2},
	\end{equation*}
	where $X_1(\cdot)$ is defined in (\ref{eq:x1def}).
\end{lemma}
Lemma \ref{lemm:DouglasRussell} is obtained from \cite[Lemma 1]{DouglasRussell1982} and Lemma \ref{lemm:jacobiest}-(ii). 
\subsection{Estimates of $e_h^n$ under some assumptions}\label{subsec:p2p1proofbegin}
Let
\begin{equation}\label{eq:errordef}
(e_h^n,\eps_h^n) \equiv (u_h^n-\pro{u}^n, p_h^n-\pro{p}^n), \quad
\eta (t) \equiv u(t)-\pro{u}(t),
\end{equation}
where 
$(u,p)$ is the solution of (\ref{NS}), 
$(\pro{u}(t), \pro{p}(t))$ is the Stokes projection of $(u(t),p(t))$ defined in (\ref{eq:stokesproj})
and $(u_h^n, p_h^n)$ is the solution of \zettai \ at the step $n$.
From (\ref{eq:NSweak}), (\ref{eq:stokesproj}) and (\ref{eq:scheme2}) 
we have the error equations in $(e_h^n,\eps_h^n)$:
\begin{subequations}\label{eq:errorWhole}
\begin{align}
\label{eq:errorMain}
&&	\left(\bwd e_h^n,v_h\right) 	+a(e_h^n,v_h)+b(v_h,\eps_h^n) 	&=\sum_{i=1}^4 (R_i^n,v_h), &\forall v_h \in V_h,\\
&&	b(e_h^n,q_h) &= 0, &\forall q_h \in Q_h, \label{errorMainb}
\end{align}
\end{subequations}
for $n=1,\dots,N_T$, where  
\begin{equation}\label{eq:defR}
\begin{split}
	R_{1}^n &\equiv \frac{Du^n}{Dt}
	-\frac{u^n-u^{n-1}\circ X_1(u^{n-1})}{\Delta t}, \\
	R_{2}^n &\equiv \frac{u^{n-1}\circ X_1(\schemeint u_h^{n-1})-u^{n-1}\circ X_1(u^{n-1})}{\Delta t},\\
	R_{3}^n &\equiv \frac{\eta^n-\eta^{n-1}\circ X_1(\schemeint u_h^{n-1})}{\Delta t}, \quad
	R_{4}^n \equiv -\frac{e_h^{n-1}-e_h^{n-1}\circ X_1(\schemeint u_h^{n-1})}{\Delta t}.
\end{split}
\end{equation}
\begin{lemma}\label{lemm:errorr}
Suppose Hypotheses \ref{hypo:NSreg} and \ref{hypo:inverse}. 
Under the condition 
\begin{equation}\label{eq:errorrCond}
{\Delta t\absi{u^{n-1}}_{1,\infty}, ~ \Delta t\absi{\lagk{1} u_h^{n-1}}_{1,\infty} \leq \done},
\end{equation}
it holds that
\begin{subequations}
\begin{align}
	\mynorm{R_{1}^n}_0 \leq & \be_1 \sqrt{\Delta t} \mynorm{u}_{Z^2(t^{n-1},t^n)}, \label{eq:lemmR1}\\
	\mynorm{R_{2}^n}_0 \leq & 
	\be_2 \mynorm{e_h^{n-1}}_0 + \be_3 h^2(\mixn{(u,p)^{n-1}}_{H^3\times H^2} + \snormsob{u^{n-1}}{2}{2}),
	\label{eq:lemmR2new}\\
	\mynorm{R_{3}^n}_0 \leq & \be_4 \frac{h^2}{\sqrt{\Delta t}} \Bigl(  \mixn{(u,p)}_{H^1(t^{n-1},t^n;H^{3}\times H^2)} 
	\notag \\ &
	+ \normsob{u_h^{n-1}}{0}{\infty} \mixn{(u,p)}_{L^2(t^{n-1},t^n;H^{3}\times H^2)} \Bigr),\label{eq:lemmR3new}\\
	\mynorm{R_{4}^n}_0 \leq & \be_5 \normsob{u_h^{n-1}}{0}{\infty} \abs{e_h^{n-1}}_1 \label{eq:lemmR4}
\end{align}
\end{subequations}
for $n=1,\dots ,N_T$, where 
$\be_1 = \be_1(\mynorm{u}_{C(W^{1,\infty})})$,
$\be_2 = \be_2(\abs{u}_{C(W^{1,\infty})}, \aintfour)$, 
$\be_3 = \be_3(\abs{u}_{C(W^{1,\infty})}, \ainteh, \aintfour, \astokes)$,
$\be_4 = \be_4(\astokes)$, 
$\be_5 = \sqrt 2$ 
and the notation $\be_i(A)$ means that a positive constant $\be_i$ depends on a set of parameters $A$.
\end{lemma}
\begin{proof}
We prove (\ref{eq:lemmR1}).
We decompose $R_1^n$ as follows: 
\begin{equation*}
\begin{split}
	R_1^n(x) =& 
	\biggl\{\henbi{u^{n}}{t}(x)+(u^{n-1}(x)\cdot\nabla)u^n(x) -\frac{u^n-u^{n-1}\circ X_1(u^{n-1})}{\Delta t}(x) \biggr\}\\
	&+ (u^n(x)-u^{n-1}(x))\cdot \nabla u^n(x)
	\equiv R_{11}^n(x) + R_{12}^n(x).
\end{split}
\end{equation*}
Setting 
\[  y(x,s)=x+(s-1)\Delta t ~ u^{n-1}(x), \quad t(s)=t^{n-1}+s \Delta t, 
\]
we have 
\[
	\frac{u^n-u^{n-1}\circ X_1(u^{n-1})}{\Delta t} = \frac{1}{\Delta t}\big[ u(y(\cdot,s),t(s)) \big]_{s=0}^1, 
\]
which implies that  
\begin{align*}
	R_{11}^n &= \henbi{u^n}{t} + (u^{n-1}\cdot \nabla) u^n
		- \int_{0}^{1} \left\{ (u^{n-1}(\cdot)\cdot \nabla) u + \frac{\partial u}{\partial t} \right\} (y(\cdot,s),t(s))    ds 
\\ &
=
 \Delta t \int_{0}^{1} ds \int_{s}^{1} \left\{ \left( u^{n-1}(\cdot)\cdot \nabla  + \frac{\partial}{\partial t} \right)^2 u  \right\} (y(\cdot,s_1),t(s_1)) d s_1 
\\ &
= \Delta t  \int_{0}^{1} s_1 \left\{ \left( u^{n-1}(\cdot)\cdot \nabla  + \frac{\partial}{\partial t} \right)^2 u  \right\}  (y(\cdot,s_1),t(s_1))    d s_1 . 
\end{align*}
Hence, we have 
\begin{align*}
\| R_{11}^n \|_0 
& \le \Delta t  \int_{0}^{1} s_1 
  \left\| 
       \left\{ \left( u^{n-1}(\cdot)\cdot \nabla  + \frac{\partial}{\partial t} \right)^2 u  \right\}  (y(\cdot,s_1),t(s_1)) 
  \right\|_0   d s_1  \nonumber
\\
& \le \be_1'(\mynorm{u}_{C(L^{\infty})}) \sqrt{\Delta t} \mynorm{u}_{Z^2(t^{n-1},t^n)}, 
 \label{est:r1}
\end{align*}
where we have used the transformation of independent variables from $x$ to $y$ and $s_1$ to $t$ and 
the estimate 
$| \det(\partial x / \partial y) | \le 2$ by virtue of Lemma \ref{lemm:jacobiest}-(ii).  
It is easy to show
\begin{equation*}
	\normsobi{R_{12}^n}{0}{2} \leq \sqrt{\Delta t} \snormsobi{u^n}{1}{\infty} \mynorm{\henbi{u}{t}}_{L^2(t^{n-1},t^n;L^2)}.
\end{equation*}
From the triangle inequality we get (\ref{eq:lemmR1}).

We prove (\ref{eq:lemmR2new}).
Using Lemma \ref{lemm:twofunc} with $q=2$, $p=1$, $p'=\infty$, $w_1=\schemeint u_h^{n-1}$, $w_2=u^{n-1}$ and $\psi = u^{n-1}$, 
we have
\begin{equation*}\label{eq:R3est1}
\begin{split}
	\normi{R_{2}^n}_0 
	&\leq \abs{u^{n-1}}_{1,\infty} \normsobi{\schemeint u_h^{n-1}-u^{n-1}}{0}{2} \\
	&\leq 
	\abs{u^{n-1}}_{1,\infty}(\normsobi{\schemeint u_h^{n-1}-\schemeint u^{n-1}}{0}{2}+\normsobi{\schemeint u^{n-1}-u^{n-1}}{0}{2}).
\end{split}
\end{equation*}
From Lemmas \ref{lemm:intest} and \ref{lemm:stokespro}-(i) we evaluate the first term as follows: 
\begin{equation}\label{eq:interpoleR3est} %
\begin{split}
	&\normi{\schemeint u_h^{n-1} - \schemeint u^{n-1}}_0 = \normi{\schemeint (u_h^{n-1} - \schemeint u^{n-1})}_0 \\
	\leq & \aintfour \normi{u_h^{n-1} - \schemeint u^{n-1}}_0 \\
	\leq & \aintfour(\mynorm{u_h^{n-1}-\pro u^{n-1}}_0 + \mynorm{\pro u^{n-1} - u^{n-1}}_0 + \normi{u^{n-1} - \schemeint u^{n-1}}_0) \\
	\leq & \aintfour(\normi{e_h^{n-1}}_0 + \astokes h^2\mixn{(u,p)^{n-1}}_{H^3\times H^2} + \ainteh h^2\snormsob{u^{n-1}}{2}{2}).
\end{split}
\end{equation}
The second term is evaluated as follows: 
\begin{equation*}
\normsobi{\schemeint u^{n-1}-u^{n-1}}{0}{2}\leq \ainteh h^2 \snormsob{u^{n-1}}{2}{2}.
\end{equation*}
Thus, we have
\begin{equation*}\label{eq:R3est2}
\begin{split}
	\normsobi{R_2^n}{0}{2}
	\leq \snormsob{u^{n-1}}{1}{\infty} &\Bigl\{\aintfour  (\mynorm{e_h^{n-1}}_0 + \astokes h^2\mixn{(u,p)^{n-1}}_{H^3\times H^2}) \\ &+ \ainteh(1+\aintfour) h^2 \snormsob{u^{n-1}}{2}{2} \Bigr\},
\end{split}
\end{equation*}
which implies (\ref{eq:lemmR2new}).

We prove (\ref{eq:lemmR3new}). Let
\begin{equation*}
	y(x)=x+(s-1)\Delta t \lagk 1 u_h^{n-1}(x), \quad t(s)=t^{n-1}+s\Delta t.
\end{equation*}
Since $R_3^n$ is rewritten as 
\begin{equation*}
	R_3^n = \int_{0}^{1} \left\{ \henbi{\eta}{t} + (\lagk 1 u_h^{n-1}(\cdot) \cdot \nabla) \eta  \right\} (y(\cdot,s), t(s))  ds,
\end{equation*}
we have, by using the change of the variable and Lemma \ref{lemm:jacobiest}-(ii),
\begin{align*}
	\normsob{R_3^n}{0}{2} 
	&\leq \biggl\lVert\int_{0}^{1}  \left| \henbi{\eta}{t} \right| (y(\cdot,s), t(s)) ds \biggr\rVert_0 + \biggl\lVert\int_{0}^{1} |(\lagk 1 u_h^{n-1}(\cdot) \cdot \nabla) \eta|(y(\cdot,s), t(s))   ds \biggr\rVert_0  \\
	&\leq \frac{\sqrt 2}{\sqrt{\Delta t}} \biggl( \mynorm{\henbi{\eta}{t}}_{L^2(t^{n-1}, t^n; L^2)} + \normi{\lagk 1 u_h^{n-1}}_{0,\infty}\mynorm{\nabla \eta}_{L^2(t^{n-1}, t^n; L^2)}\biggr) \\
	& \leq \frac{\sqrt{2}\astokes h^2}{\sqrt{\Delta t}} \left( \mixn{(u,p)}_{H^1(t^{n-1}, t^n; H^3\times H^2)} +  \mynorm{u_h^{n-1}}_{0,\infty}\mixn{(u,p)}_{L^2(t^{n-1}, t^n; H^3\times H^2)}  \right),
\end{align*}
which implies (\ref{eq:lemmR3new}).

The inequality (\ref{eq:lemmR4}) is obtained  
from Lemma \ref{lemm:twofunc} with $q=2$, $p=\infty$, $p'=1$, $w_1=0$, $w_2=\lagk 1 u_h^{n-1}$ and $\psi = e_h^{n-1}$.
\qed
\end{proof}
\begin{lemma} \label{lemm:mainprop}
Suppose Hypotheses \ref{hypo:NSreg} and \ref{hypo:inverse}. 
Let $n \in \{1,\cdots,N_T  \}$ be any integer and let $u_h^{n-1} \in V_h$ be known.   
Suppose that $u_h^{n-1}$ satisfies 
\begin{align}\label{uhn-1}
b(u_h^{n-1}, q_h) &= 0, &\forall q_h \in Q_h.  
\end{align}     
Under the condition \eqref{eq:errorrCond}, 
there exists a solution $(u_h^{n},p_h^{n})$ of \eqref{eq:scheme2}
and
it holds that 
\begin{equation*}
\begin{split}
	&\mynorm{\bwd{e_h^n}}_0^2 + \bwd{(\nu\abs{e_h^n}_1^2)} \\
	\leq &
	\be_{21}(\normsob{u_h^{n-1}}{0}{\infty}) \nu \abs{e_h^{n-1}}_1^2 
	+ \be_{22}(\normsob{u_h^{n-1}}{0}{\infty}) 
	\Bigl\{
		\Delta t \mynorm{u}_{Z^2(t^{n-1},t^n)}^2 \\
		&+ \frac{h^4}{\Delta t}  \mixn{(u,p)}_{H^1(t^{n-1},t^n;H^{3}\times H^2)}^2 
		+ h^4 \left( \mixn{(u,p)^{n-1}}_{H^3\times H^2}^2 +\snormsob{u^{n-1}}{2}{2}^2 \right)
	\Bigr\}, 
\end{split}
\end{equation*}
where $e_h^n$ is defined in (\ref{eq:errordef}),
and $\be_{21}(\xi)$ and $\be_{22}(\xi)$ are the functions defined in (\ref{eq:funcbe}) below.
\end{lemma}
\begin{proof}
Since it holds that $\Delta t\absi{\lagk{1} u_h^{n-1}}_{1,\infty} \leq \done$, the mapping $X_1(\lagk 1 u_h^{n-1}):\Omega \to \Omega$ is bijective from Lemma \ref{lemm:jacobiest}-(i).
Hence, there exists a solution $(u_h^n, p_h^n)$ of \eqref{eq:scheme2}. 
Substituting $v_h = \bwd{e_h^n}$ in (\ref{eq:errorMain}), we have
\begin{equation}\label{eq:erroreqSubstitute}
	\normsob{\bwd{e_h^n}}{0}{2}^2 + \bwd{\left(\frac{\nu}{2} \normsob{\nabla e_h^n}{0}{2}^2\right)} + b(\bwd{e_h^n}, \eps_h^n) \leq \sum_{i=1}^4 (R_{i}^n, \bwd{e_h^n}).  
\end{equation}
From \eqref{uhn-1} and \eqref{eq:scheme2} the term $b(\bwd{e_h^n}, \eps_h^n)$ of the left-hand side vanishes.  
Using Schwarz' and Young's inequalities and Lemma \ref{lemm:errorr}, we have 
\begin{equation*}
\begin{split}
	&\normsob{\bwd{e_h^n}}{0}{2}^2 + \bwd{\left(\frac{\nu}{2} \abs{e_h^n}_1^2\right)} 
	\leq 2 \biggl\{ \be_1^2 \Delta t  \mynorm{u}_{Z^2(t^{n-1},t^n)}^2 \\
	&+\left( \be_2 \mynorm{e_h^{n-1}}_0 + \be_3 h^2(\mixn{(u,p)^{n-1}}_{H^3\times H^2} + \snormsob{u^{n-1}}{2}{2}) \right)^2
	\\ 
	&+ \be_4^2 \frac{h^4}{\Delta t} \left(  \mixn{(u,p)}_{H^1(t^{n-1},t^n;H^{3}\times H^2)} + \normsob{u_h^{n-1}}{0}{\infty} \mixn{(u,p)}_{L^2(t^{n-1},t^n;H^{3}\times H^2)} \right)^2\\
	& 
	+ \be_5^2 \normsob{u_h^{n-1}}{0}{\infty}^2 \abs{e_h^{n-1}}_1^2 \biggr\}
	+ \frac{1}{2} \normsob{\bwd e_h^n}{0}{2}^2,
\end{split}
\end{equation*}
which implies that 
\begin{equation*}
\begin{split}
	&\bwd{(\nu\abs{e_h^n}_1^2)} + \normsob{\bwd{e_h^n}}{0}{2}^2 \\
	\leq &
	\be_{11} \Bigl(\mynorm{e_h^{n-1}}_0^2 
		+ \mynorm{u_h^{n-1}}_{0,\infty}^2 \abs{e_h^{n-1}}_1^2 \Bigr)
	+ \be_{12} 
	\Bigl\{
	 \Delta t \mynorm{u}_{Z^2(t^{n-1},t^n)}^2 \\
	&+\frac{h^4}{\Delta t} ( \mixn{(u,p)}_{H^1(t^{n-1},t^n;H^{3}\times H^2)}^2 
	+ \normsob{u_h^{n-1}}{0}{\infty}^2 \mixn{(u,p)}_{L^2(t^{n-1},t^n;H^{3}\times H^2)}^2) \\
	&
	+ h^4 \left( \mixn{(u,p)^{n-1}}_{H^3\times H^2}^2 + \snormsob{u^{n-1}}{2}{2}^2 \right) 
	\Bigr\}, 
\end{split}
\end{equation*}
where $\be_{11}$ and $\be_{12}$ 
are constants depending only on $\be_{1},\dots, \be_{5}$.
Using Poincar\'e's inequality $\normsobi{e_h^{n-1}}{0}{2} \leq \apoin \snormsobi{e_h^{n-1}}{1}{2}$ and
defining the functions $\be_{21}$ and $\be_{22}$ by
\begin{equation}\label{eq:funcbe}
	\be_{21}(\xi) = \frac{\be_{11}}{\nu} (\apoin^2 + \xi^2), ~
	\be_{22}(\xi) = \be_{12} (1+\xi^2), 
\end{equation}
we have the conclusion.
\qed
\end{proof}
\subsection{Definitions of constants $\cprop$, $\czero$ and $\htheo$}
We first define constants $\beone$ and $\betwo$ by
\begin{equation*}
	\beone \equiv \be_{21}(\mynorm{u}_{C(L^\infty)}+1), ~
	\betwo \equiv \be_{22}(\mynorm{u}_{C(L^\infty)}+1).
\end{equation*}
We define two positive constants $\cprop$ and $\czero$ by
\begin{equation*}
\begin{split}
	\cprop \equiv &
	\big\{\nu^{-1}(1+\apoin^2)\exp(\beone T)\betwo\big\}^{1/2} \max
	\Bigl\{ \mynorm{u}_{Z^2}, 
	\Bigl( 
	\mixn{(u,p)}_{H^1(H^{3}\times H^2)}^2 \\ & 
	+ T\left(\mixn{(u,p)}_{C(H^3\times H^2)}^2 +\abs{u}_{C(H^2)}^2\right) 
	+\nu \astokes^2 \absi{p^0}_2^2 
	\Bigr)^{1/2} \Bigr\}
\end{split}
\end{equation*}
and
\begin{align}\label{c0}
	\czero \equiv 
		\donemodify{
		\frac{1}{4}\sqrt{\frac{1}{\cuint \ainvone \cprop}}
		}.
\end{align}
Let a positive constant $\htwo$ be small enough to satisfy that 
\begin{subequations} \label{eq:hrestrictWhole}
\begin{align}
	\ainvzero \htwo^{1-d/6} \left( \cprop \htwo + \astokes \htwo \mixn{(u,p)}_{C(H^3\times H^2)} + \ainteh \abs{u}_{C(H^2)} \right) & \notag \\ +  \ainvzero \cprop \czero \htwo^{d/12} &\leq 1, 
	\label{eq:hrestrictA}\\
	\czero \Bigl\{
	\ainvone \htwo^{1-d/4} \left( \cprop \htwo + \astokes \htwo \mixn{(u,p)}_{C(H^3\times H^2)} + \ainteh \abs{u}_{C(H^2)} \right) & \notag \\ + \aintone \htwo^{d/4} \abs{u}_{C(W^{1,\infty})} 
	\Bigr\}
	&\leq 
		\donemodify{\frac{3}{16\cuint}}.
	\label{eq:hrestrictB}
\end{align}
\end{subequations}
which are possible since all the powers of $h_0$ are positive.  
\subsection{Induction}\label{subsec:p2p1proofend}
For $n=0,\dots, N_T$ we define the property P$(n)$ by 
\begin{description}
\item[(a)]
	$ \displaystyle
	\nu\abs{e_h^n}_1^2 + \mynorm{\bwd{e_h}}_{\ell^2(1,n;L^2)}^2 
	\leq  \exp(\beone n\Delta t)
	\betwo \Bigl\{\Delta t^2 \mynorm{u}_{Z^2(t^{0},t^n)}^2 \\ + h^4(  \mixn{(u,p)}_{H^1(t^{0},t^n;H^{3}\times H^2)}^2 
	+ \mixn{(u,p)}_{\ell^2(0,n-1;H^3\times H^2)}^2
	+ \abs{u}_{\ell^2(0,n-1;H^2)}^2 ) \\
	+ \nu \abs{e_h^0}_1^2 \Bigr\}.
	$
\item[(b)]
	$\mynorm{u_h^n}_{0,\infty} \leq \mynorm{u}_{C(L^\infty)}+1$. 
\item[(c)]
	$\Delta t\absi{\lagk{1} u_h^{n}}_{1,\infty} \leq \done$.
\end{description}
\renewcommand{\proofname}{Proof of Theorem \ref{theo:mainConvergence}}
\begin{proof}
We first prove that P($n$) holds for $n=0,\dots,N_T$ by induction.  
When $n=0$, the property P(0)-(a) obviously holds with the equality.
The properties P(0)-(b) and (c) are proved in similar ways to and easier than P($n$)-(b) and (c) below, we omit the proofs.  

Let $n \in \{1,\cdots,N_T \}$ be any integer.   
Supposing that P($k$), $k=1,\dots,n-1$, holds true, we prove that P($n$) holds.
We now apply Lemma \ref{lemm:mainprop}.  
The condition \eqref{uhn-1} is satisfied trivially when $n \ge 2$.  
When $n=1$, from the choice of $u_h^0$, (\ref{eq:stokesproj}) and Remark \ref{divu00} we have 
\begin{equation}\label{eq:b0vanish}
	b(e_h^0, q_h)=b(u_h^0,q_h)-b(\pro{u}^0,q_h)=0-0=0,  \quad \forall q_h \in Q_h. 
\end{equation}
We consider the condition \eqref{eq:errorrCond}.  
The former condition follows from
$\Delta t \leq \czero h^{d/4}$ and (\ref{eq:hrestrictB}) by the inequality
\begin{equation*}
	\Delta t\absi{u^{n-1}}_{1,\infty} \leq \czero \htheo ^{d/4} \absi{u}_{C(W^{1,\infty})} 
	\leq \donemodify{\frac{3}{16\cuint}} 
	\leq \done,
\end{equation*}
and the latter condition  
$\Delta t\absi{\lagk{1} u_h^{n-1}}_{1,\infty} \leq \done$ 
follows from P($n-1$)-(c).
Hence, there exists a solution $(u_h^{n},p_h^{n})$ at the step $n$.

We begin the proof of P($n$)-(a).
By putting 
\begin{align*}
	x_n \equiv & \nu \abs{e_h^n}_1^2, \quad y_n \equiv \mynorm{\bwd e_h^n}_0^2, \\
	b_n \equiv & \Delta t \mynorm{u}_{Z^2(t^{n-1},t^n)}^2 \\ & + h^4 \left( \tfrac{1}{\Delta t} \mixn{(u,p)}_{H^1(t^{n-1}, t^n; H^{3}\times H^2)}^2 
	+ \mixn{(u,p)^{n-1}}_{H^3\times H^2}^2
	+ \abs{u^{n-1}}_{2}^2 \right),
\end{align*}
P($n$)-(a) is rewritten as
\begin{equation}\label{eq:rewritten}
	x_n+\Delta t \sum_{i=1}^n y_i 
	\leq \exp(\beone n \Delta t) \betwo \biggl( x_0 + \Delta t \sum_{i=1}^n b_i \biggr).
\end{equation}
On the other hand, Lemma \ref{lemm:mainprop} implies that
\begin{equation*}
	x_n + \Delta t y_n \leq (1+\beone \Delta t) x_{n-1} + \betwo\Delta t b_n, 
\end{equation*}
where we have used the inequalities $\be_{2i}(\normsob{u_h^{n-1}}{0}{\infty}) \leq \be_{2i}^*$, $i=1,2$, obtained from P($n-1$)-(b). 
Using the inequalities $1\leq 1+x \leq \exp(x)$ for $x\geq 0$ and P($n-1$)-(a) rewritten by \eqref{eq:rewritten}, we have 
\begin{equation*}
\begin{split}
	&x_n+\Delta t \sum_{i=1}^n y_i = x_n + \Delta t y_n + \Delta t \sum_{i=1}^{n-1} y_i \\
	\leq & (1+\beone \Delta t) x_{n-1} + \betwo \Delta t b_n + \Delta t\sum_{i=1}^{n-1} y_i \\
	\leq & (1+\beone \Delta t) \exp(\beone (n-1)\Delta t) \betwo \biggl( x_0 + \Delta t \sum_{i=1}^{n-1} b_i \biggr) + \betwo \Delta t b_n \\
	\leq & \exp(\beone n \Delta t) \betwo \biggl( x_0 + \Delta t \sum_{i=1}^n b_i \biggr),
\end{split}
\end{equation*}
which is nothing but P($n$)-(a).  

Since $u_h^0$ is the first component of $\Pi_h^S(u^0,0)$, we have 
\begin{equation*}
	e_h^0=u_h^0-\pro u^0=(\Pi_h^S(0,-p^0))_1=((\Pi_h^S-I)(0,-p^0))_1,
\end{equation*}
which implies 
$\abs{e_h^0}_1 \leq \astokes h^2 \snormsob{p^0}{2}{2}$.  
From P(0)-(a) and the definition of $\cprop$, we have
\begin{equation}\label{eq:cprop}
	\mynorm{e_h^n}_1 \leq \cprop (h^2+\Delta t).
\end{equation}

P($n$)-(b) is proved as follows:  
\begin{align}
	&\mynorm{u_h^n}_{0,\infty} \notag \\
	\leq & \normi{u_h^n-\lagk 1 u^n}_{0,\infty} + \normi{\lagk 1 u^n}_{0,\infty} 
	\notag \\ 
	\leq & \ainvzero h^{-d/6} \normi{u_h^n -\lagk 1 u^n }_{1} +  \mynorm{u^n}_{0,\infty} 
	\tag{by Lemmas \ref{lemm:inv} and \ref{lemm:intest}}\\ 
	\leq & \ainvzero h^{-d/6} (\mynorm{u_h^n-\pro{u}^n}_{1} +  \mynorm{\pro{u}^n - u^n }_{1} + \normi{u^n -\lagk 1 u^n }_{1}) +  \mynorm{u^n}_{0,\infty} \notag \\
	\leq & \ainvzero h^{-d/6} \left(\cprop(h^2+\Delta t) + \astokes h^2 \mixn{(u,p)^n}_{H^3\times H^2} + \ainteh h \snormsob{u^n}{2}{2}\right) +  \mynorm{u^n}_{0,\infty}
	\tag{by (\ref{eq:cprop}), Lemma \ref{lemm:stokespro}-(i) and Lemma \ref{lemm:intest}}\\
	\leq &  \ainvzero h^{1-d/6} \left( \cprop h + \astokes h \mixn{(u,p)^n}_{H^3\times H^2} + \ainteh \abs{u^n}_2 \right)+  \ainvzero \cprop \czero h^{d/12} + \mynorm{u^n}_{0,\infty} 
	\tag{since $\Delta t \leq \czero h^{d/4}$}\\
	\leq & 1+\mynorm{u}_{C(L^\infty)}. 
	\tag{since $h\leq \htheo$ and by (\ref{eq:hrestrictA})}
\end{align}
We prove P($n$)-(c). 
We can estimate $\abs{u_h^n}_{1,\infty} \Delta t$ as follows: 
\begin{align}
	&\abs{u_h^n}_{1,\infty} \Delta t \notag \\
	\leq & (\absi{u_h^n -\lagk 1 u^n }_{1,\infty} + \absi{\lagk 1 u^n}_{1,\infty})\Delta t 
	\notag\\ 
	\leq & \{\ainvone h^{-d/2} (\absi{u_h^n -\lagk 1 u^n }_{1}) + \aintone \abs{u^n}_{1,\infty}\}\Delta t 
	\tag{by Lemmas \ref{lemm:inv} and \ref{lemm:intest}}\\ 
	\leq & \{\ainvone h^{-d/2} (\abs{u_h^n-\pro{u}^n}_{1} +  \abs{\pro{u}^n - u^n }_{1} + \absi{u^n -\lagk 1 u^n }_{1}) + \aintone \abs{u^n}_{1,\infty}\}\Delta t 
	\notag\\
	\leq & \Bigl\{\ainvone h^{-d/2} \left(\cprop(h^2+\Delta t) + \astokes h^2 \mixn{(u,p)^n}_{H^3\times H^2} + \ainteh h \snormsob{u^n}{2}{2}\right) \notag \\ &+ \aintone \abs{u^n}_{1,\infty} \Bigr\}\Delta t 
	\tag{by (\ref{eq:cprop}), Lemma \ref{lemm:stokespro}-(i) and Lemma \ref{lemm:intest}}\\
	\leq &  \czero\left\{\ainvone h^{1-d/4} \left( \cprop h + \astokes h \mixn{(u,p)^n}_{H^3\times H^2} + \ainteh \abs{u^n}_2 \right)+ \aintone h^{d/4} \abs{u^n}_{1,\infty} \right\} \notag \\ &+ \ainvone \cprop \czero^2
	\tag{since $\Delta t \leq \czero h^{d/4}$}\\
	\leq & \donemodify{\frac{3}{16\cuint} + \frac{1}{16\cuint}= \frac{1}{4\cuint}.}
	\tag{since $h\leq \htheo$, and by (\ref{eq:hrestrictB}) and the definition of $\czero$}
\end{align}
From this estimate and the definition of $\cuint$, we have $\Delta t\absi{\lagk{1} u_h^{n}}_{1,\infty} \leq \done$.

Thus, we have proved that P($n$) holds for $n=0,\cdots,N_T$. 

From P($n$)-(a), $n=0,\dots,N_T$, we obtain 
\begin{equation*}
	\mynorm{e_h}_{\ell^\infty(H^1)}, \mynorm{\bwd e_h}_{\ell(1,N_T; L^2)} \leq \cprop(h^2+\Delta t).
\end{equation*}
Using the triangle inequality $\mynorm{u_h-u}_{\ell^\infty(H^1)} \leq \mynorm{e_h}_{\ell^\infty(H^1)} + \mynorm{\eta}_{\ell^\infty(H^1)}$, we get  
\begin{equation*}
\mynorm{u_h-u}_{\ell^\infty(H^1)} \leq \ctheoone(h^2+\Delta t).
\end{equation*}

We now prove the estimate on the pressure.  
We can evaluate $\eps_h^n$ as follows: 
\begin{align}
	\normsob{\eps_h^n}{0}{2} 
	\leq & \frac{1}{\beinfsup}\sup_{v_h\in V_h}\frac{b(v_h,\eps_h^n )}{\mynorm{v_h}_1} \tag{by Lemma \ref{lemm:infsup}}\\
	= & \frac{1}{\beinfsup}\sup_{v_h\in V_h}\frac{1}{\mynorm{v_h}_1}\biggl(\sum_{i=1}^4(R_{i}^n,v_h)-(\bwd e_h^n, v_h)-a(e_h^n,v_h)\biggr) \tag{by (\ref{eq:errorMain})}\\
	\leq & \frac{1}{\beinfsup}\biggl( \sum_{i=1}^4 \mynorm{R_{i}^n}_0 + \mynorm{\bwd e_h^n}_0 + \nu \abs{e_h^n}_1 \biggr) \notag \\
	\leq & c \biggl( \mynorm{\bwd e_h^n}_0 + \nu \abs{e_h^n}_1 + \normsobi{e_h^{n-1}}{1}{2} +\sqrt{\Delta t} \mynorm{u}_{Z^2(t^{n-1},t^n)} + h^2 \mixn{(u,p)^{n-1}}_{H^3\times H^2} \notag\\
	&+\frac{h^2}{\sqrt{\Delta t}}   \mixn{(u,p)}_{H^1(t^{n-1},t^n;H^{3}\times H^2)} + h^2 \snormsob{u^{n-1}}{2}{2} \biggr) \tag{by Lemma \ref{lemm:errorr} and P($n-1$)-(b)}, 
\end{align}
which implies that, from \eqref{eq:cprop}, 
\begin{equation*}
	\mynorm{\eps_h}_{\ell^2(L^2)} \leq c(\mynorm{\bwd e_h}_{\ell^2(L^2)} + h^2+\Delta t) \leq c(h^2+\Delta t),
\end{equation*}
where $c$ is a positive constant independent of $h$ and $\Delta t$.
Using the triangle inequality 
\begin{equation*}
	\mynorm{p_h-p}_{\ell^2(L^2)} \leq \mynorm{\eps_h}_{\ell^2(L^2)} + \mynorm{p-\pro{p}}_{\ell^2(L^2)},
\end{equation*}
we obtain $\mynorm{p_h-p}_{\ell^2(L^2)} \leq \ctheoone(h^2+\Delta t)$.
\qed
\end{proof}
\subsection{Proof of Theorem \ref{theo:mini}}\label{subsec:miniproof}
In this subsection we prove the result on the $\mini$-element.
At first we replace the estimates of $R_{2}^n$ and $R_{3}^n$ in Lemma \ref{lemm:errorr}.  
\begin{lemmMini}
Suppose Hypotheses \hypoMiniNum \ and \ref{hypo:inverse}. 
Under the condition \eqref{eq:errorrCond}
it holds that 
\begin{align*}
	\mynorm{R_{2}^n}_0 &\leq 
	\be_2 \mynorm{e_h^{n-1}}_0 + \be_3 (h\mixn{(u,p)^{n-1}}_{H^2\times H^1} + h^2\snormsob{u^{n-1}}{2}{2}),\\
	\mynorm{R_{3}^n}_0 &\leq \be_4 \frac{h}{\sqrt{\Delta t}} \left(  \mixn{(u,p)}_{H^1(t^{n-1},t^n;H^{2}\times H^1)} + \normsob{u_h^{n-1}}{0}{\infty} \mixn{(u,p)}_{L^2(t^{n-1},t^n;H^{2}\times H^1)} \right)
\end{align*}
for $n=1,\dots ,N_T$.  
\end{lemmMini}
The proof is similar to Lemma \ref{lemm:errorr} by replacing the order $k=2$ by $k=1$ in Lemma \ref{lemm:stokespro}-(i). 
\renewcommand{\proofname}{Proof of Theorem \ref{theo:mini}}
\begin{proof}
We only show the outline of the proof for the existence of $(u_h, p_h)$ and the inequality (\ref{eq:miniEst1}) since the proof is similar to that of Theorem \ref{theo:mainConvergence}.
We replace the definition of $\cprop$ by
\begin{equation*}
\begin{split}
	\cprop \equiv &
	\big\{\nu^{-1}(1+\apoin^2)\exp(\beone T)\betwo \big\}^{1/2} \max\Big\{ \mynorm{u}_{Z^2}, 
	\Big( \mixn{(u,p)}_{H^1(H^{2}\times H^1)}^2 \\& 
	+ T \left( \mixn{(u,p)}_{C(H^2\times H^1)}^2  +\abs{u}_{C(H^2)}^2 \right) 
	+\nu \astokes^2 \absi{p^0}_1^2 
	\Big)^{1/2}\Big\}, 
\end{split}
\end{equation*}
redefine $\czero$ by \eqref{c0} with the new $\cprop$, 
and replace the condition (\ref{eq:hrestrictWhole}) on $\htheo$ by
\begin{subequations} \label{eq:hrestrictMiniWhole}
\begin{align}
	&\ainvzero \htwo^{1-d/6} \Bigl( \cprop + \astokes \mixn{(u,p)}_{C(H^2\times H^1)} + \ainteh \abs{u}_{C(H^2)} \Bigr)  +  \ainvzero \cprop \czero \htheo^{d/12} \leq 1, \\
	&\czero \Bigl\{
	\ainvone \htheo^{1-d/4} \Bigl( \cprop + \astokes \mixn{(u,p)}_{C(H^2\times H^1)} + \ainteh \abs{u}_{C(H^2)} \Bigr)   \notag \\ & \quad + \aintone \htheo^{d/4} \abs{u}_{C(W^{1,\infty})} 
	\Bigr\} 
	\leq 
	\donemodify{\frac{3}{16\cuint}}. 
\end{align}
\end{subequations}
We also replace P($n$)-(a) by 
\begin{equation*}
\begin{split}
	&\nu\abs{e_h^n}_1^2 + \mynorm{\bwd{e_h}}_{\ell^2(1,n;L^2)}^2 
	\leq  \exp(\beone n\Delta t)
	\betwo \Bigl\{\Delta t^2 \mynorm{u}_{Z^2(t^{0},t^n)}^2 \\ & + h^2
	\left(  \mixn{(u,p)}_{H^1(t^{0},t^n;H^{2}\times H^1)}^2  
	+ \mixn{(u,p)}_{\ell^2(0,n-1;H^2\times H^1)}^2
	+ \abs{u}_{\ell^2(0,n-1;H^2)}^2 
	\right)
	+ \nu \abs{e_h^0}_1^2 \Bigr\}.
\end{split}
\end{equation*}
P($n$)-(a) implies the estimate 
\begin{align}\label{h1ofehn}
  \normsob{e_h^n}{1}{2}\leq \cprop(h+\Delta t). 
\end{align}
The choice (\ref{eq:hrestrictMiniWhole}) is sufficient to derive P($n$)-(b) and (c).
Hence, the existence of the solution and the estimate (\ref{eq:miniEst1}) are obtained similarly.  

We now prove the estimate (\ref{eq:miniEst2}), following \cite{Suli} except the introduction of $X_1(\Pi_h^{(1)}u_h^{n-1})$.  
Substituting $(v_h,q_h)=(e_h^n,\eps_h^n)$ in (\ref{eq:errorWhole}), we have
\begin{equation}\label{eq:miniEstErrorEq}
	\frac{1}{2} \bwd \normsob{e_h^n}{0}{2}^2 + \frac{1}{2\Delta t} \normsob{e_h^n-e_h^{n-1}}{0}{2}^2+\nu \snormsob{e_h^n}{1}{2}^2
	= \sum_{i=1}^4(R_i^n,e_h^n),
\end{equation}
where $R_i$, $i=1,\cdots,4$, are defined in (\ref{eq:defR}).
The term $(R_1^n,e_h^n)$ is evaluated by \eqref{eq:lemmR1}.  
From Lemma \ref{lemm:stokespro}-(ii) we have 
\[
  \normsob{\pro u^{n-1} - u^{n-1}}{0}{2}\leq \aon h^2 \mixn{(u,p)^{n-1}}_{H^2\times H^1}. 
\]
Using this estimate in the last line in \eqref{eq:interpoleR3est}, 
we have 
\begin{equation*}
	(R_2^n, e_h^n) \leq \normsob{R_{2}^n}{0}{2} \normsob{e_h^n}{0}{2} 
	\leq \{
	\be_2 \mynorm{e_h^{n-1}}_0 + \be_3' h^2 (\mixn{(u,p)^{n-1}}_{H^2\times H^1} + \snormsob{u^{n-1}}{2}{2})\}
	\normsob{e_h^n}{0}{2}.
\end{equation*}
We divide the term $(R_3^n,e_h^n)$ as follows: 
\begin{equation*}\label{eq:estR2}
\begin{split}
	(R_3^n,e_h^n)
	=&\frac{1}{\Delta t}(\eta^n-\eta^{n-1},e_h^n) 
	+ \frac{1}{\Delta t}(\eta^{n-1}-\eta^{n-1}\circ X_1(\lagk 1 u^{n-1}),e_h^n) \\
	&+ \frac{1}{\Delta t}(\eta^{n-1}\circ X_1(\lagk 1 u^{n-1})-\eta^{n-1}\circ X_1(\lagk 1 u_h^{n-1}),e_h^n)\\
	&\equiv I_1+I_2+I_3.
\end{split}
\end{equation*}
The first term $I_1$ is evaluated as
\begin{equation*}
	I_1 
	\leq \frac{1}{\sqrt{\Delta t}}\norm{\henbi{\eta}{t}}_{L^2(t^{n-1},t^n;L^2)}\normsob{e_h^n}{0}{2}
	\leq \frac{\aon h^2}{\sqrt{\Delta t}} \mixn{(u,p)}_{H^1(t^{n-1},t^n;H^2\times H^1)} \normsob{e_h^n}{0}{2}.
\end{equation*}
By Lemma \ref{lemm:DouglasRussell} the second term $I_2$ is evaluated as  
\begin{equation*}
\begin{split}
	I_2
	\leq & \adr\aintone\mynorm{u}_{C(W^{1,\infty})} \normsob{\eta^{n-1}}{0}{2} \normsob{e_h^n}{1}{2} \\
	\leq & \adr\aintone\mynorm{u}_{C(W^{1,\infty})} \aon h^2\mixn{(u,p)^{n-1}}_{H^2\times H^1} \normsob{e_h^n}{1}{2}.
\end{split}
\end{equation*}
In order to evaluate $I_3$ we prepare the estimate 
\[
	\ainvzero h^{-d/6} \snormsob{\eta^{n-1}}{1}{2} \le 
	\ainvzero h^{1-d/6} \astokes   \mixn{(u,p)^{n-1}}_{H^2\times H^1} \le 1, 
\]
where we have used Lemma \ref{lemm:stokespro}-(i) and (\ref{eq:hrestrictMiniWhole}a).  

Using Lemma \ref{lemm:twofunc} with $q=1$, $p=p'=2$, $w_1=\lagk 1 u^{n-1}$, $w_2=\lagk 1 u_h^{n-1}$ and $\psi=\eta^{n-1}$, 
Lemma \ref{lemm:inv}, the above estimate and (\ref{eq:interpoleR3est}), 
we can evaluate $I_3$ as follows: 
\begin{equation*}
\begin{split}
	I_3
	\leq & \frac{1}{\Delta t} \normsobi{\eta^{n-1}\circ X_1(\lagk 1 u^{n-1})-\eta^{n-1}\circ X_1(\lagk 1 u_h^{n-1})}{0}{1} \normsob{e_h^n}{0}{\infty} 
	\\
	\leq & \sqrt 2 \ainvzero h^{-d/6} \snormsob{\eta^{n-1}}{1}{2} \normsobi{\lagk 1 u^{n-1}-\lagk 1 u_h^{n-1}}{0}{2} \normsobi{e_h^n}{1}{2} 
	\\
	\leq & \sqrt 2  \normsobi{\lagk 1 u^{n-1}-\lagk 1 u_h^{n-1}}{0}{2} \normsobi{e_h^n}{1}{2} 
\\
	\leq & \sqrt 2 \aintfour(\normi{e_h^{n-1}}_0 + \aon h^2\mixn{(u,p)^{n-1}}_{H^2\times H^1} + \ainteh h^2\snormsob{u^{n-1}}{2}{2}) \normsob{e_h^n}{1}{2}.
\end{split}
\end{equation*}
In order to evaluate $(R_4^n,e_h^n)$ we prepare the estimate 
\[
	\ainvzero h^{-d/6} \snormsob{e_h^{n-1}}{1}{2} \le 
	\ainvzero h^{-d/6} \cprop (h+\Delta t) \le 
	\ainvzero \cprop h^{-d/6} (h+\czero h^{d/4}) \le 1, 
\]
where we have used \eqref{h1ofehn} and (\ref{eq:hrestrictMiniWhole}a).  
Using Lemma \ref{lemm:DouglasRussell}, the above inequality and a similar estimate to $I_3$ in $(R_3^n,e_h^n)$, we can evaluate $(R_4^n,e_h^n)$as follows: 
\begin{equation*}
\begin{split}
	(R_4^n,e_h^n)
	=&-\frac{1}{\Delta t}(e_h^{n-1}-e_h^{n-1}\circ X_1(\lagk 1 u^{n-1}),e_h^n) \\
	&- \frac{1}{\Delta t}(e_h^{n-1}\circ X_1(\lagk 1 u^{n-1})-e_h^{n-1}\circ X_1(\lagk 1 u_h^{n-1}),e_h^n) \\
	\leq & \adr\aintone\mynorm{u}_{C(W^{1,\infty})} \normsobi{e_h^{n-1}}{0}{2} \normsob{e_h^n}{1}{2} \\& + \sqrt 2 \aintfour(\normi{e_h^{n-1}}_0 + \aon h^2\mixn{(u,p)^{n-1}}_{H^2\times H^1} + \ainteh h^2\snormsob{u^{n-1}}{2}{2}) \normsob{e_h^n}{1}{2}.
\end{split}
\end{equation*}
Combining (\ref{eq:miniEstErrorEq}) with these estimates and using Young's inequality and Poincar\'e's inequality $\normsob{e_h^n}{0}{2} \leq \apoin \snormsob{e_h^n}{1}{2}$, we have, 
\begin{equation*}
\begin{split}
	&\frac{1}{2}\bwd \normsob{e_h^n}{0}{2}^2 + \nu \snormsob{e_h^n}{1}{2}^2 
	\leq \be_{31} \Bigl( \normsob{e_h^{n-1}}{0}{2} + \sqrt{\Delta t} \mynorm{u}_{Z^2(t^{n-1},t^n)} \\
	&+ \frac{h^2}{\sqrt{\Delta t}} \mixn{(u,p)}_{H^1(t^{n-1},t^n;H^2\times H^1)} + h^2 \mixn{(u,p)^{n-1}}_{H^2\times H^1} + h^2 \snormsob{u^{n-1}}{2}{2}
	\Bigr)\normsob{e_h^n}{1}{2} \\
	\leq & \nu \snormsob{e_h^n}{1}{2}^2 + \be_{32} \normsob{e_h^{n-1}}{0}{2}^2 + \be_{33} \Bigl\{ \Delta t \mynorm{u}_{Z^2(t^{n-1},t^n)}^2 + \frac{h^4}{\Delta t} \mixn{(u,p)}_{H^1(t^{n-1},t^n;H^2\times H^1)}  \\
	&+ h^4 \left(\mixn{(u,p)^{n-1}}_{H^2\times H^1}^2 +  \snormsob{u^{n-1}}{2}{2}^2 \right) \Bigr\},
\end{split}
\end{equation*}
where 
$\be_{31}$, $\be_{32}$ and $\be_{33}$ are positive constants independent of $h$ and $\Delta t$.  
Applying Gronwall's inequality, we obtain (\ref{eq:miniEst2}).
\qed
\end{proof}
\renewcommand{\proofname}{Proof}

\section{Numerical results}\label{sec:numex}
We show numerical results in $d=2$ for the $\pk2/\pk1$-element.  
We compare the conventional \jurai with the present \zettai.
For the triangulation of the domain the FreeFem++ \cite{FreeFemCite} is used.  
In \jurai we employ numerical quadrature of seven-point formula of degree five \cite{HMS}. 
The relative error $E_X$ is defined by
\begin{equation*}
	E_X(\phi) \equiv \frac{\mynorm{\Pi_h \phi-\phi_h}_X}
	{\mynorm{\Pi_h\phi}_X},
\end{equation*}
for $\phi=u$ in $X=\ell^\infty(H_0^1)$ and $\ell^\infty(L^2)$, and for $\phi=p$ in $X=\ell^2(L^2)$.

\begin{exam}\label{exam:funcC}
	In (\ref{NS}), let 	$\Omega \equiv (0,1)^2$, $T=1$.   
	We consider the two cases, $\nu=10^{-2}$ and $10^{-4}$. 
	The functions $f$ and $u^0$ are defined so that the exact solution is 
	\begin{equation*}
		\begin{split}
			u_1(x,t)&=\phi(x_1,x_2,t),\\
			u_2(x,t)&= - \phi(x_2,x_1,t),\\
			p(x,t)&=\sin(\pi(x_1 + 2x_2) + 1 + t),
		\end{split}
	\end{equation*}
	where $\phi(a,b,t)\equiv -\sin(\pi a)^2 \sin(\pi b) \{\sin(\pi(a + t)) + 3\sin(\pi(a + 2b + t))\}$.
\end{exam}

Let $N$ be the division number of each side of $\Omega$.
We set $h\equiv 1/N$. 
Figure \ref{fig:irsq_renum} shows the triangulation of $\bar \Omega$ for $N=16$. 
The time increment $\Delta t$ is set to be $\Delta t=h^2$ ($N=16, 23, 32 ,45$ and $64$) or $\Delta t=h^3$ ($N=16, 19 ,23, 27$ and $32$) so that we can observe the convergence behavior of order $h^2$ or $h^3$. 
The purpose of the choice $\Delta t=O(h^2)$ or $O(h^3)$ is to examine the theoretical convergence order, but it is not based on the stability condition, which is much weaker as shown in Theorem \ref{theo:mainConvergence}.  
\begin{figure}
	\centering
	\includegraphics[width=35mm]{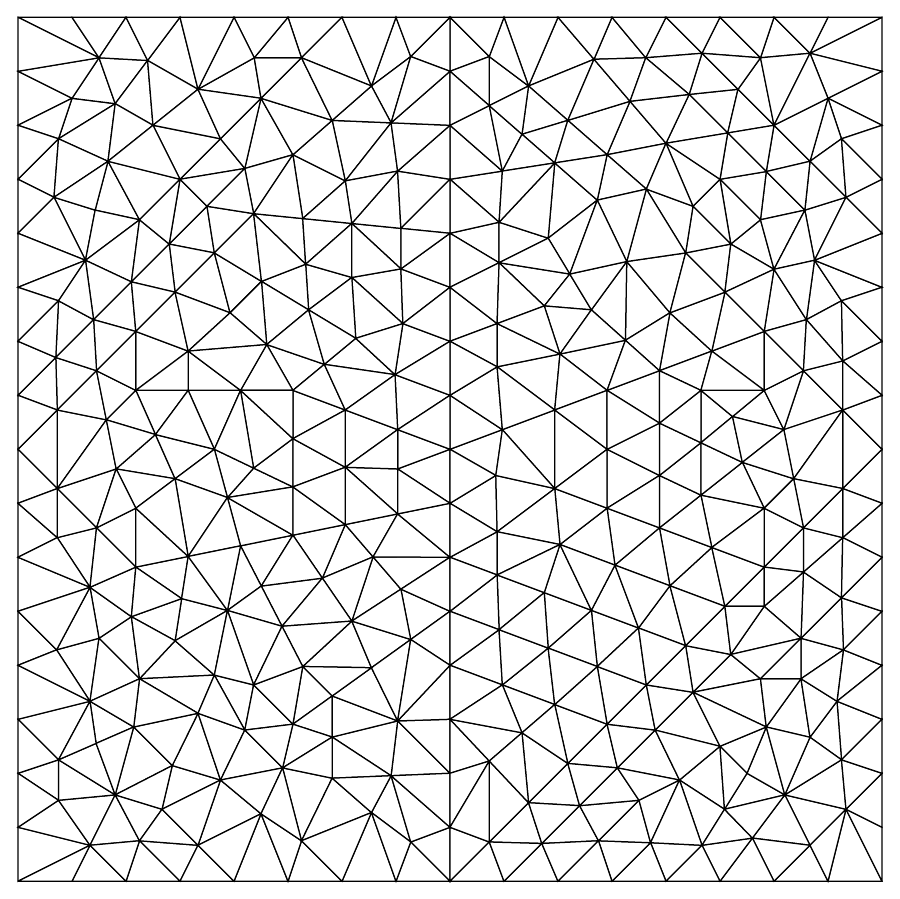}
	\caption{
		The triangulation of $\bar \Omega$ for $N=16$
	}
	\label{fig:irsq_renum}
\end{figure}
\begin{table}[h]
	\centering
	\caption{Symbols used in Figs. \ref{fig:funcCloglognu-2} and \ref{fig:funcCloglognu-4}, and Tables \ref{table:funcCnu-2} and \ref{table:funcCnu-4}
	}
	\begin{tabular}{ccccc}
		\hline
		$\phi$ & $u$ & $p$ & $u$ & $u$ \\
		$X$ & $\lih$ & $\ltl$ & $\lil$ & $\lil$ \\
		$\Delta t$ & $h^2$ & $h^2$ & $h^2$ & $h^3$ \\
		\hline
		\hline
		\jurai	 & \mone & \mtwo & \mthree & \mfour \rule{0cm}{\columnheight}\\
		\zettai	 & \mfive & \msix & \mseven & \meight \rule{0cm}{\columnheight}\\
		\hline
	\end{tabular}
	\label{table:symbol}
\end{table}

\begin{figure}[h]
	\begin{overpic}[width=\figwidthbasic 
		]
		{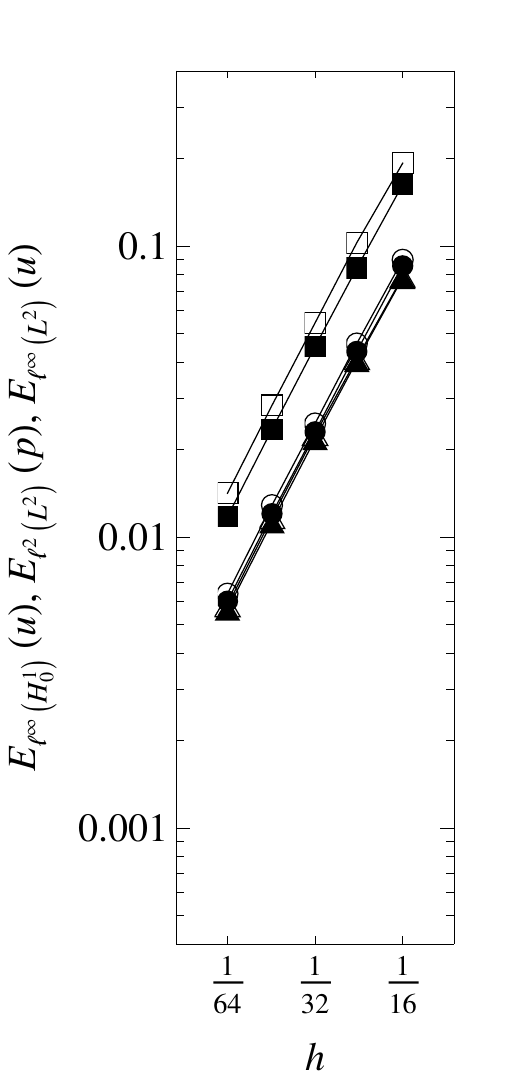}
		\put(30,15){\includegraphics[width=9mm]{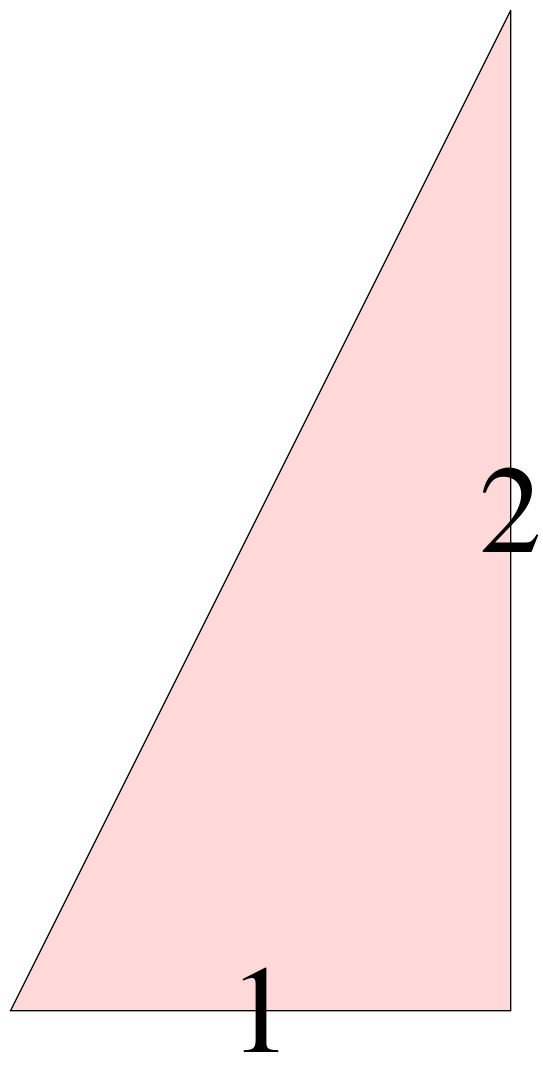}}
	\end{overpic} 
	\hspace{17mm}
	\begin{overpic}[width=\figwidthbasic 
		]
		{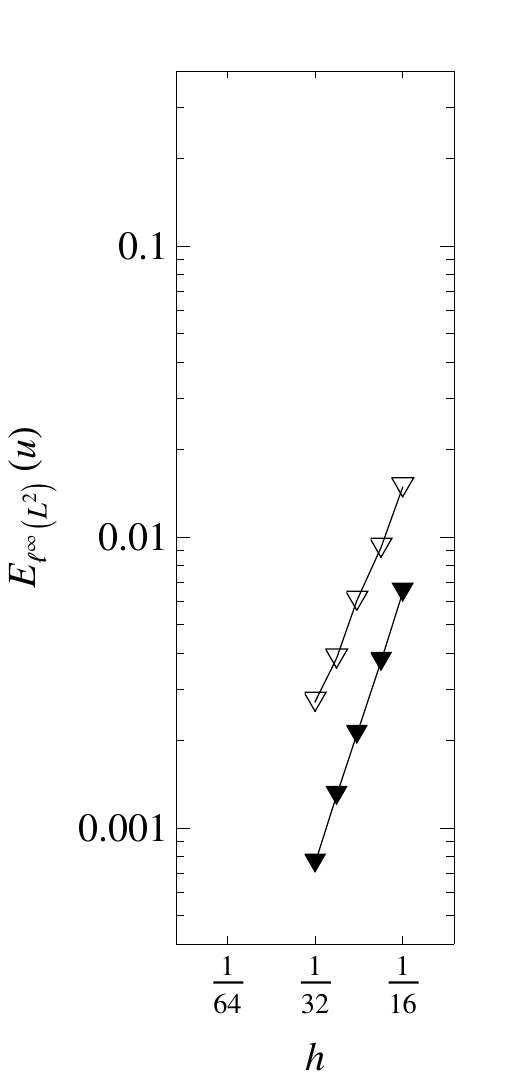}
		\put(30,15){\includegraphics[width=9mm]{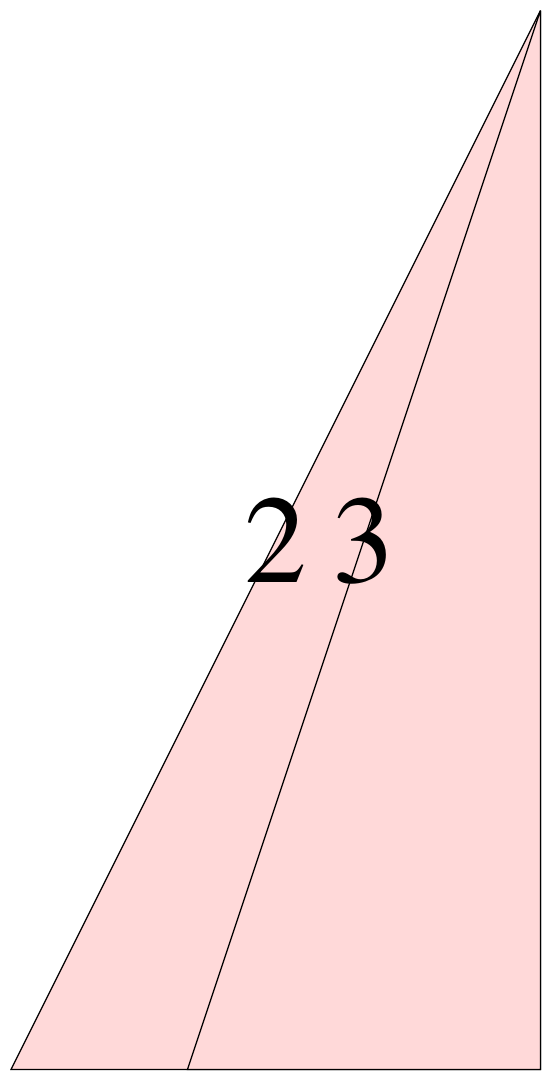}}
	\end{overpic}
	\caption{
	Relative errors $E_{\ell^\infty(H^1_0)}(u)$, $E_{\ell^2(L^2)}(p)$, $E_{\ell^\infty(L^2)}(u)$ with $\Delta t=h^2$ (left) and  $E_{\ell^\infty(L^2)}(u)$  with $\Delta t=h^3$ (right) in the case of $\nu=10^{-2}$ in Example \ref{exam:funcC}
	}
	\label{fig:funcCloglognu-2}
\end{figure}
\begin{table}[h]
	\caption{
	The values of relative errors and orders in Fig. \ref{fig:funcCloglognu-2} by \jurai (top) and \zettai (bottom)
	}
	\label{table:funcCnu-2}
	\begin{tabular}{lllllll}
		\hline
		$N$ & \mone & order & \mtwo & order & \mthree & order \rule{0cm}{\columnheight}\\
		\hline
		\hline
		16 & 8.55e-2 &  & 1.63e-1 &  & 7.77e-2 & \\
		23 & 4.34e-2 & 1.87  & 8.40e-2 & 1.82  & 4.03e-2 & 1.81 \\
		32 & 2.30e-2 & 1.93  & 4.52e-2 & 1.88  & 2.17e-2 & 1.87 \\
		45 & 1.20e-2 & 1.90  & 2.34e-2 & 1.92  & 1.13e-2 & 1.93 \\
		64 & 6.02e-3 & 1.97  & 1.18e-2 & 1.96  & 5.64e-3 & 1.96 \\
		
		\hline
		\\
		\hline
		$N$ & \mfive & order & \msix & order & \mseven & order \rule{0cm}{\columnheight}\\
		\hline
		\hline
		16 & 8.97e-2 &  & 1.93e-1 &  & 7.84e-2 & \\
		23 & 4.62e-2 & 1.83  & 1.03e-1 & 1.73  & 4.10e-2 & 1.78 \\
		32 & 2.46e-2 & 1.92  & 5.44e-2 & 1.92  & 2.25e-2 & 1.82 \\
		45 & 1.29e-2 & 1.90  & 2.84e-2 & 1.91  & 1.17e-2 & 1.93 \\
		64 & 6.39e-3 & 1.99  & 1.41e-2 & 1.97  & 5.81e-3 & 1.98 \\
		
		\hline
	\end{tabular}
	\begin{tabular}{lll}
		\hline
		$N$ & \mfour & order \rule{0cm}{\columnheight}\\
		\hline
		\hline
		16 & 6.45e-3 & \\
		19 & 3.73e-3 & 3.19 \\
		23 & 2.10e-3 & 3.02 \\
		27 & 1.29e-3 & 3.02 \\
		32 & 7.57e-4 & 3.15 \\
		
		\hline
		\\
		\hline
		$N$ & \meight & order \rule{0cm}{\columnheight}\\
		\hline
		\hline
		16 & 1.48e-2 & \\
		19 & 9.19e-3 & 2.78 \\
		23 & 6.04e-3 & 2.19 \\
		27 & 3.83e-3 & 2.85 \\
		32 & 2.72e-3 & 2.01 \\
		
		\hline
	\end{tabular}
\end{table}
\begin{figure}[h]
	\begin{overpic}[width=\figwidthbasic 
		]
		{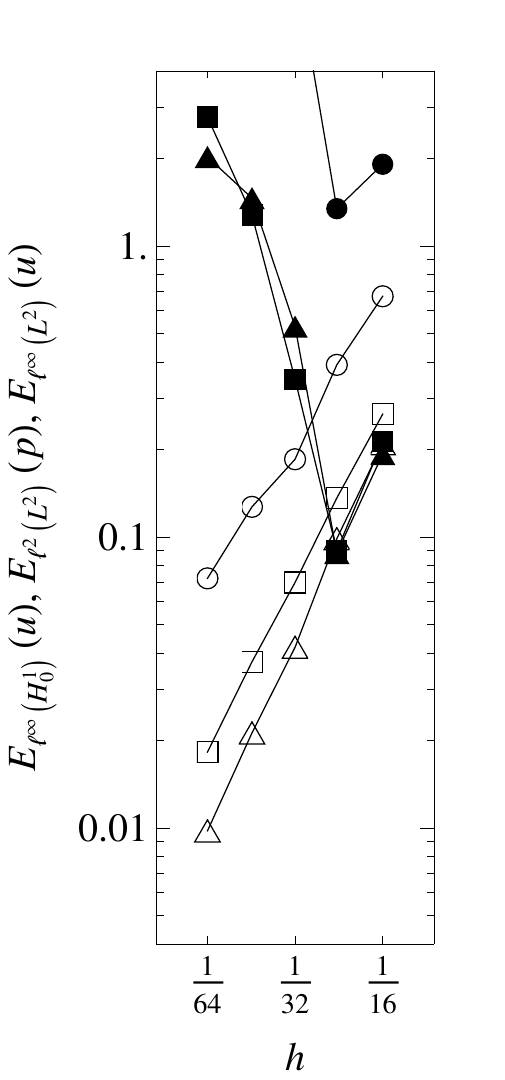}
		\put(27,15){\includegraphics[width=9mm]{tri2.pdf}}
	\end{overpic} 
	\hspace{17mm}
	\begin{overpic}[width=\figwidthbasic 
		]
		{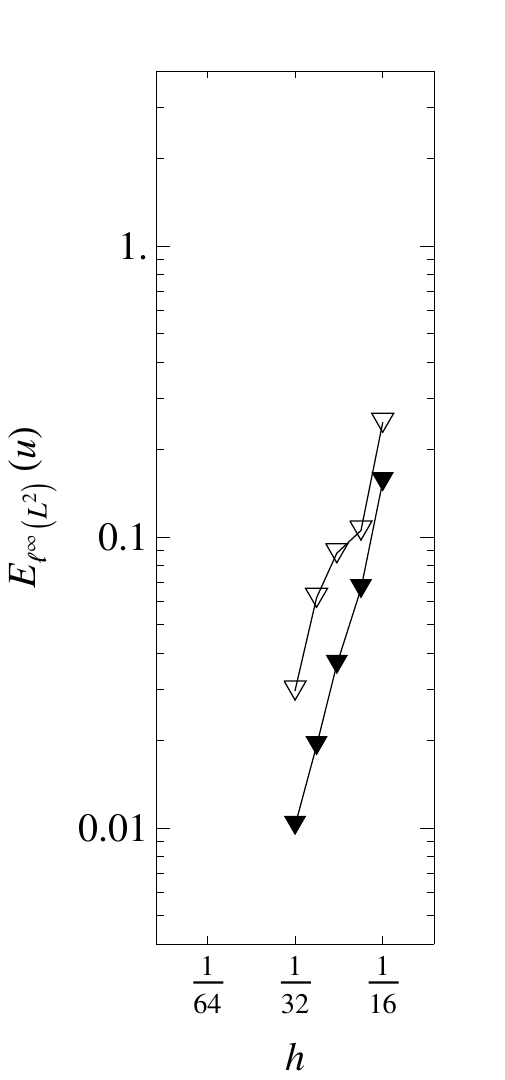}
		\put(27,15){\includegraphics[width=9mm]{trimid23.pdf}}
	\end{overpic}
	\caption{
	Relative errors $E_{\ell^\infty(H^1_0)}(u)$, $E_{\ell^2(L^2)}(p)$, $E_{\ell^\infty(L^2)}(u)$ with $\Delta t=h^2$ (left) and  $E_{\ell^\infty(L^2)}(u)$ with $\Delta t=h^3$ (right) in the case of $\nu=10^{-4}$ in Example \ref{exam:funcC}
	}
	\label{fig:funcCloglognu-4}
\end{figure}
\begin{table}[h]
	\caption{The values of errors and orders of the graphs in Fig. \ref{fig:funcCloglognu-4} with $\Delta t = h^2$ (top) and $\Delta t = h^3$ (bottom)}
	\label{table:funcCnu-4}
	\begin{tabular}{lllllll}
		\hline
		$N$ & \mone & order & \mtwo & order & \mthree & order \rule{0cm}{\columnheight}\\
		\hline
		\hline
		16 & 1.91e+0 &  & 2.14e-1 &  & 1.93e-1 & \\
		23 & 1.34e+0 & 0.97  & 8.97e-2 & 2.39  & 8.81e-2 & 2.16 \\
		32 & 9.42e+0 & -5.90  & 3.48e-1 & -4.11  & 5.28e-1 & -5.43 \\
		45 & 4.10e+1 & -4.31  & 1.28e+0 & -3.81  & 1.46e+0 & -2.98 \\
		64 & 8.82e+1 & -2.18  & 2.77e+0 & -2.20  & 2.02e+0 & -0.93 \\
		
		\hline
		\\
		\hline
		$N$ & \mfive & order & \msix & order & \mseven & order \rule{0cm}{\columnheight}\\
		\hline
		\hline
		16 & 6.72e-1 &  & 2.65e-1 &  & 2.09e-1 & \\
		23 & 3.91e-1 & 1.50  & 1.36e-1 & 1.83  & 9.88e-2 & 2.07 \\
		32 & 1.85e-1 & 2.26  & 6.98e-2 & 2.02  & 4.18e-2 & 2.60 \\
		45 & 1.27e-1 & 1.10  & 3.73e-2 & 1.84  & 2.12e-2 & 1.99 \\
		64 & 7.21e-2 & 1.61  & 1.83e-2 & 2.03  & 9.78e-3 & 2.20 \\
		
		\hline
	\end{tabular}
	\begin{tabular}{lll}
		\hline
		$N$ & \mfour & order \rule{0cm}{\columnheight}\\
		\hline
		\hline
		16 & 1.55e-1 & \\
		19 & 6.64e-2 & 4.92 \\
		23 & 3.65e-2 & 3.14 \\
		27 & 1.92e-2 & 4.01 \\
		32 & 1.02e-2 & 3.71 \\
		
		\hline
		\\
		\hline
		$N$ & \meight & order \rule{0cm}{\columnheight}\\
		\hline
		\hline
		16 & 2.47e-1 & \\
		19 & 1.05e-1 & 4.96 \\
		23 & 8.80e-2 & 0.94 \\
		27 & 6.18e-2 & 2.20 \\
		32 & 2.97e-2 & 3.29 \\
		
		\hline
	\end{tabular}
\end{table}

Table \ref{table:symbol} shows the symbols used in the graphs and tables. 
Since every graph of the relative error $E_X$ versus $h$ is depicted in the logarithmic scale, the slope corresponds to the convergence order.    
Figure \ref{fig:funcCloglognu-2} shows the graphs of $E_{\ell^\infty(H^1_0)}(u)$, $E_{\ell^2(L^2)}(p)$ and $E_{\ell^\infty(L^2)}(u)$ versus $h$ in the case of $\nu=10^{-2}$.  
Their values and convergence orders are listed in Table \ref{table:funcCnu-2}.
When $\Delta t=h^2$, the convergence orders of $\elih(u)$ (\mone, \mfive), $\eltl(p)$ (\mtwo, \msix) and $\elil(u)$  (\mthree, \mseven) are almost 2 in both schemes. 
When $\Delta t=h^3$, the order of $\elil(u)$ is almost 3 in \jurai \ (\mfour) and 2 in \zettai \ (\meight).  
They reflect the theoretical results.

We consider a higher Reynolds number case. 
	Figure \ref{fig:funcCloglognu-4} shows the graphs in the case of $\nu=10^{-4}$ and their values are listed in Table \ref{table:funcCnu-4}.
When $\Delta t=h^2$, all errors increase abnormally at $N=32, 45$ and $64$ in \jurai \ (\mone, \mtwo, \mthree) 
while the convergence is observed in \zettai \ (\mfive, \msix, \mseven) but the order of $\elih(u)$ (\mfive) is less than 2.
In order to obtain the theoretical convergence order $O(h^2)$ in \zettai, it seems that finer meshes will be necessary.
When $\Delta t=h^3$, the order of $\elil(u)$ is more than 3 in \jurai \ (\mfour) while it is less than 3 between $N=19$ and $23$, and $N=23$ and $27$ in \zettai \ (\meight).

\onefig{
	\includegraphics[height=45mm]{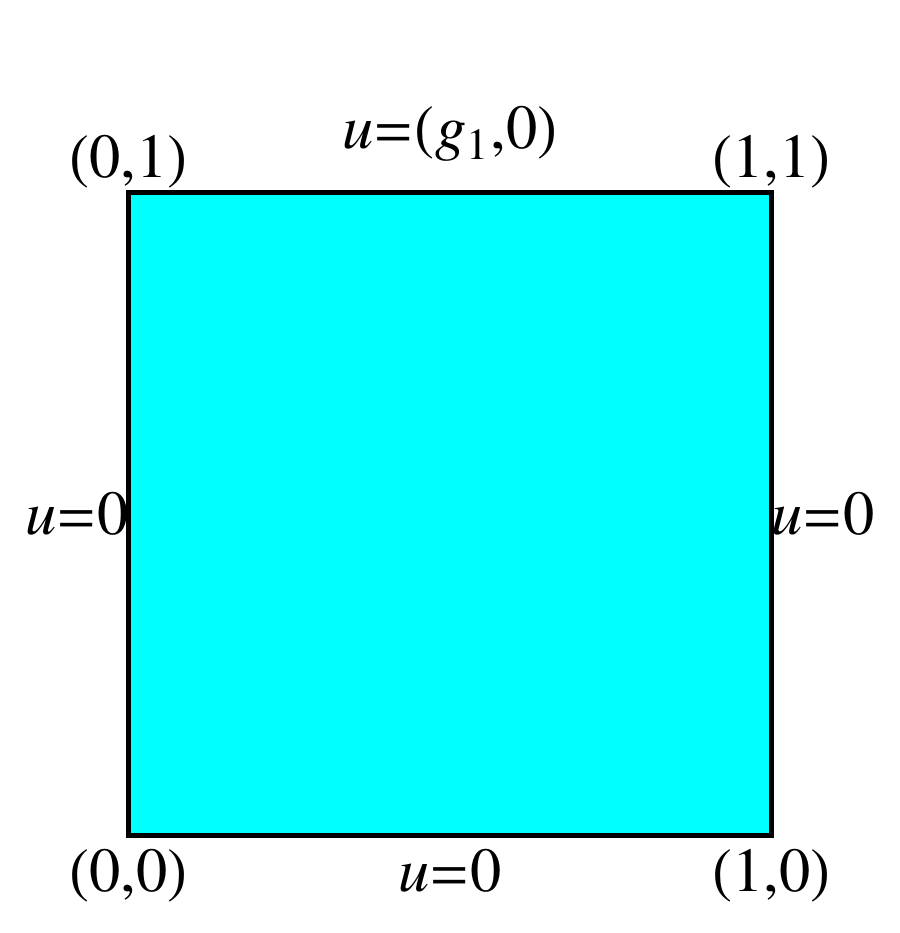}
	\quad
	\includegraphics[height=45mm]{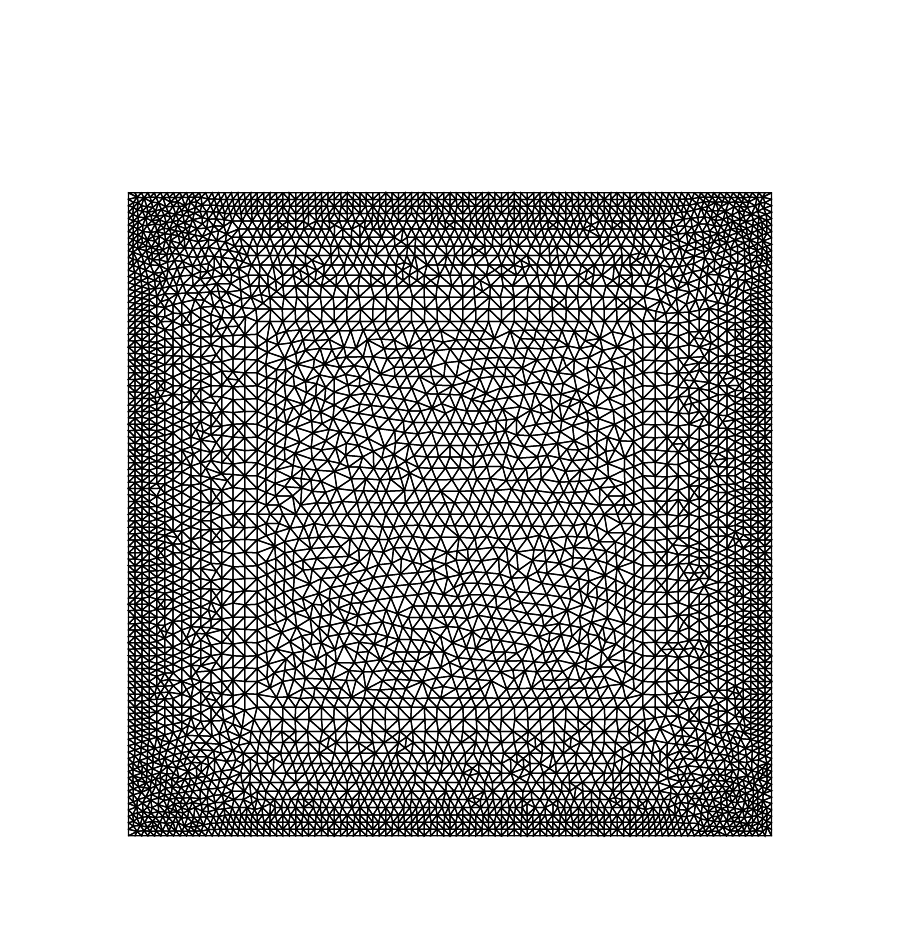}
	\caption{The domain $\Omega$ and the boundary condition (left) and the triangulation of $\Omega$ (right) in Example \ref{exam:cav2}}.
	\label{fig:cavityBoundaryAndMesh}
}

We now consider a cavity problem to see that \zettai \ is robust for high Reynolds number while \jurai \ is not.
This problem is not a homogeneous Dirichlet boundary problem, but it is often used as a benchmark problem.  
In order to assure the existence of the solution we deal with a regularized cavity problem, where the prescribed velocity is continuous on the boundary.   
\begin{exam}\label{exam:cav2}
	Let $\Omega \equiv (0,1)^2$, $f=0$, $u^0=0$.  We consider the two cases, $\nu=10^{-4}$ and $10^{-5}$. 
	The boundary condition is described in Fig. \ref{fig:cavityBoundaryAndMesh} (left), where $g_1=4x_1(1-x_1)$. 
\end{exam}

\onefig{
	\includegraphics[width=34mm]{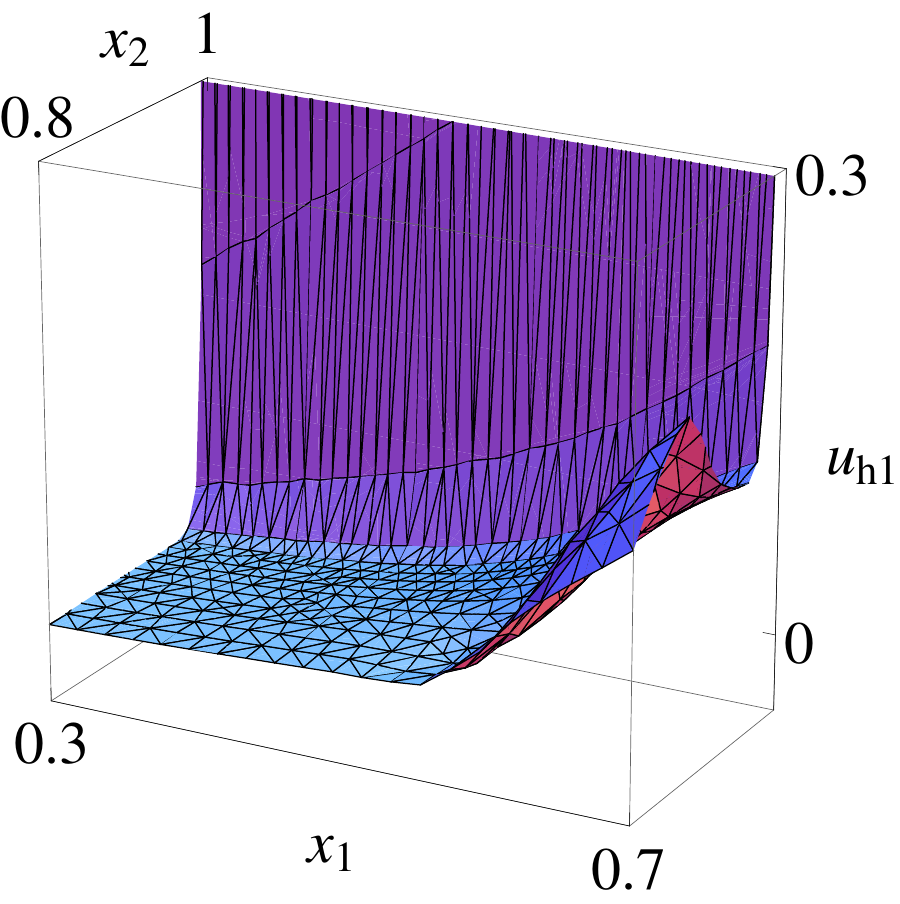} 
	\quad
	\includegraphics[width=34mm]{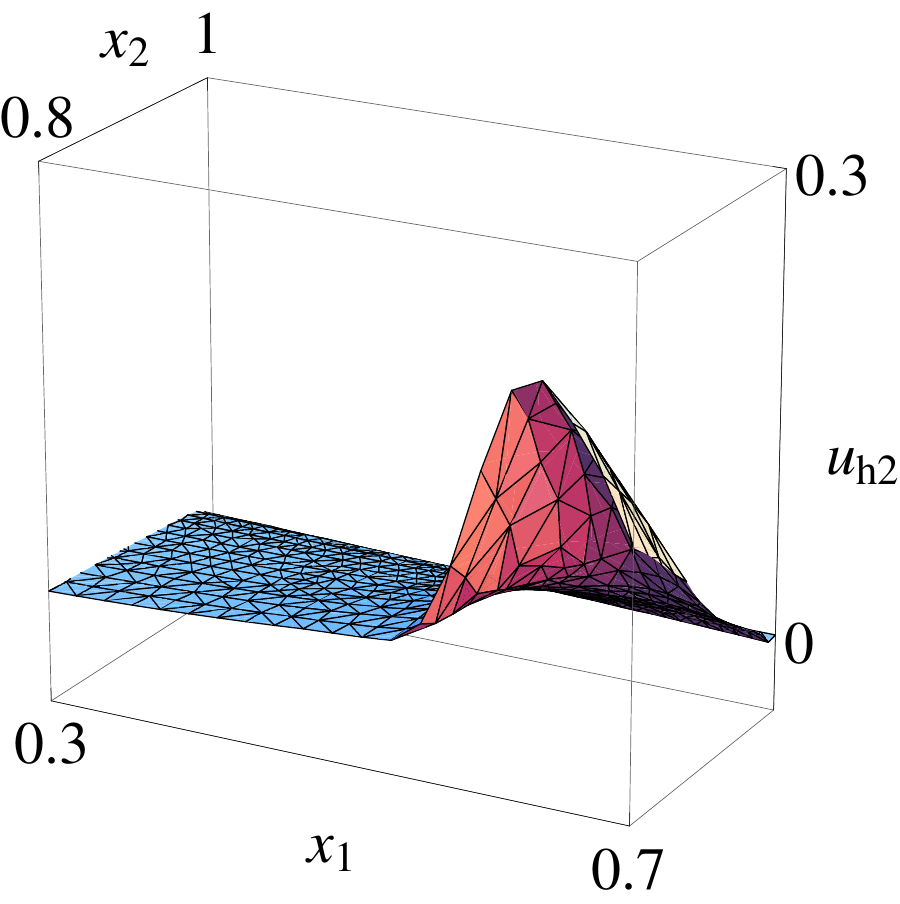} 
	\caption{
	Stereographs of $u_{h1}^n$ (left) and $u_{h2}^n$ (right) at $t^n=8$ by \jurai \ in Example \ref{exam:cav2} when $\nu=10^{-4}$ 
	}
	\label{fig:cavnu-4sch1}
}
\onefig{
	\includegraphics[width=34mm]{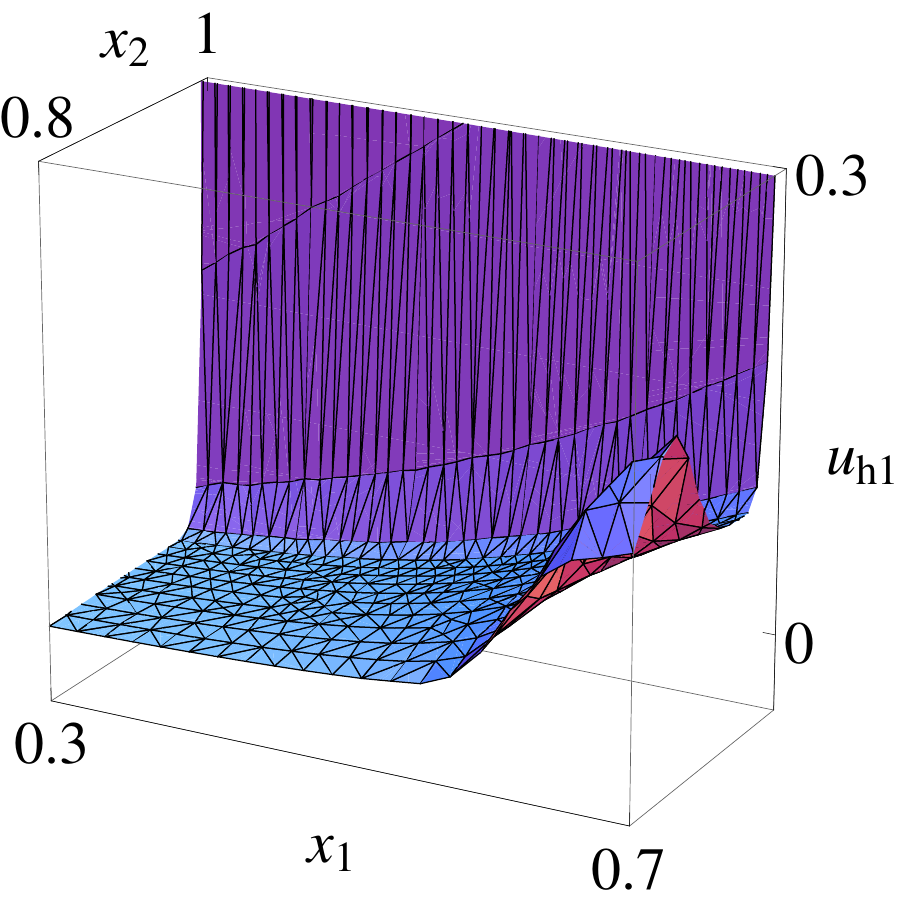} 
	\quad
	\includegraphics[width=34mm]{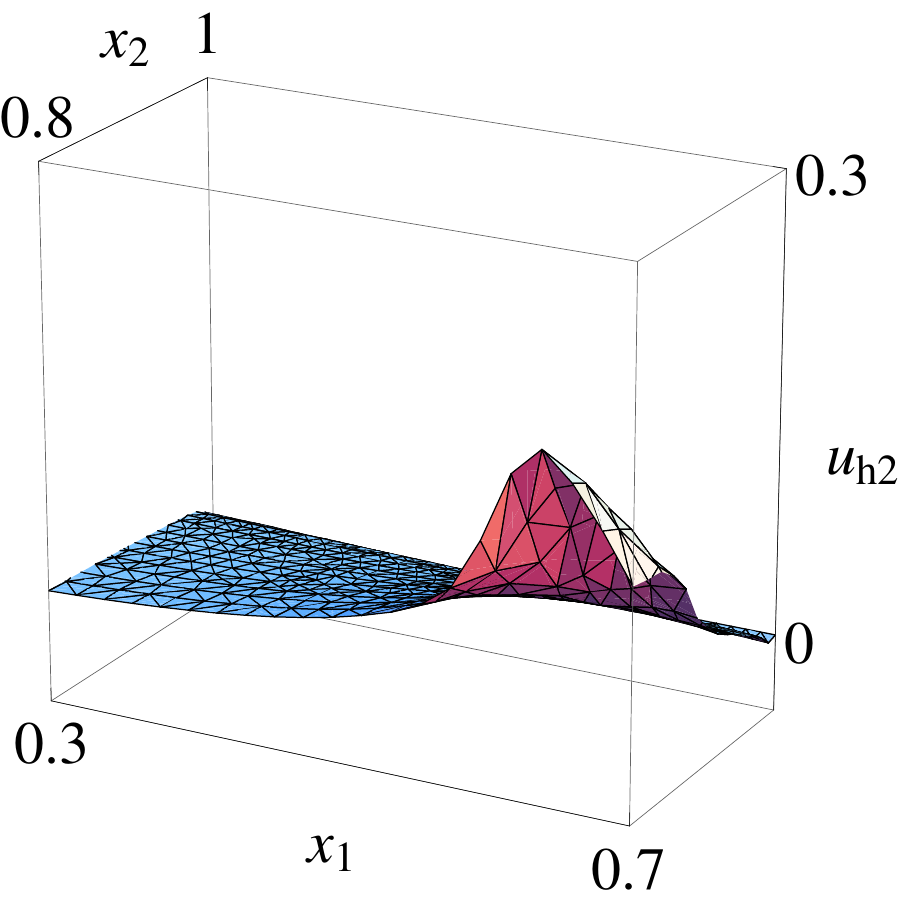} 
	\caption{
	Stereographs of $u_{h1}^n$ (left) and $u_{h2}^n$ (right) at $t^n=8$ by \zettai \ in Example \ref{exam:cav2} when $\nu=10^{-4}$ 
	}
	\label{fig:cavnu-4sch2}
}
\onefig{
	\includegraphics[width=34mm]{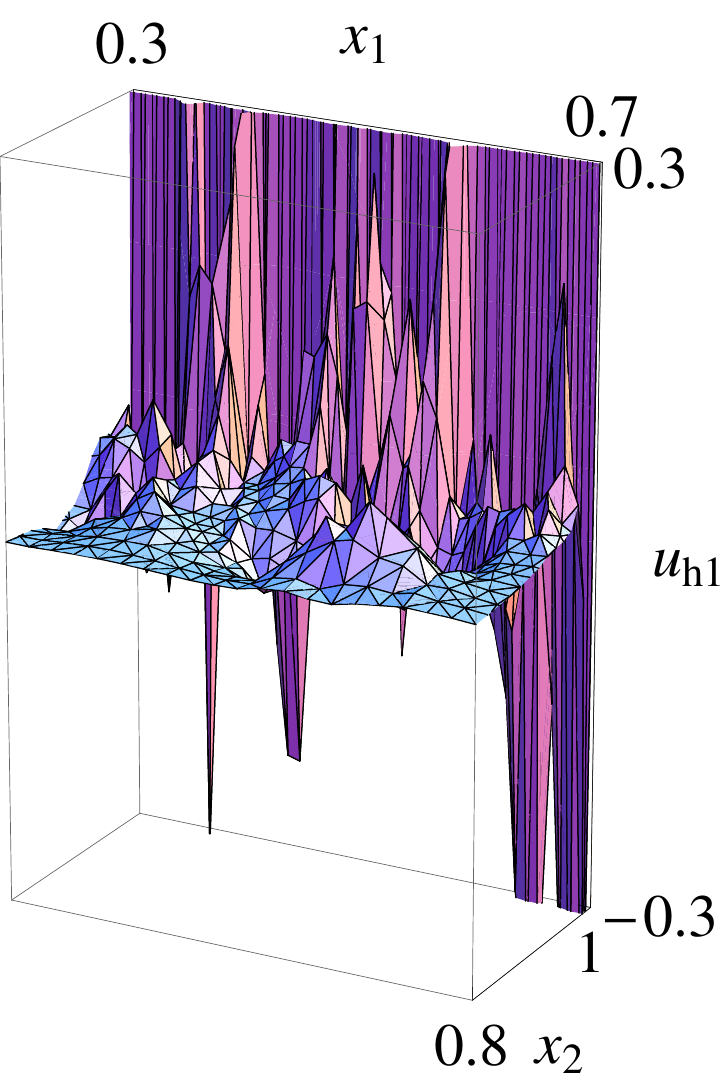} 
	\quad
	\includegraphics[width=34mm]{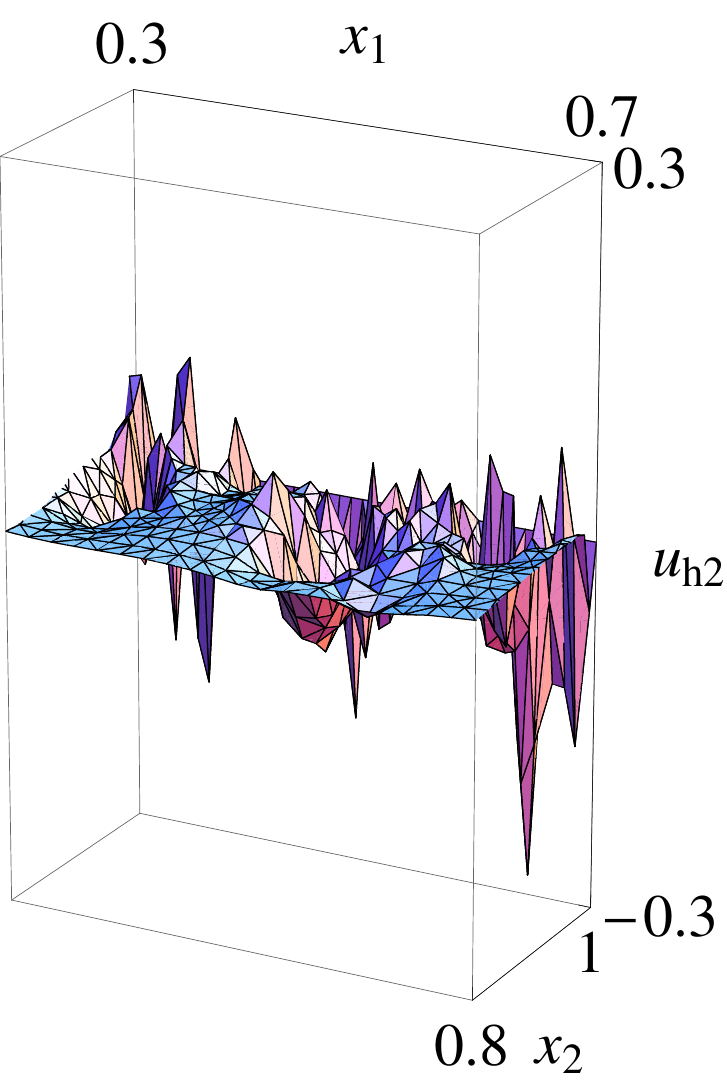} 
	\caption{
	Stereographs of $u_{h1}^n$ (left) and $u_{h2}^n$ (right) at $t^n=8$ by \jurai \ in Example \ref{exam:cav2} when $\nu=10^{-5}$ 
	}
	\label{fig:cavnu2-5sch1}
}
\onefig{
	\includegraphics[width=34mm]{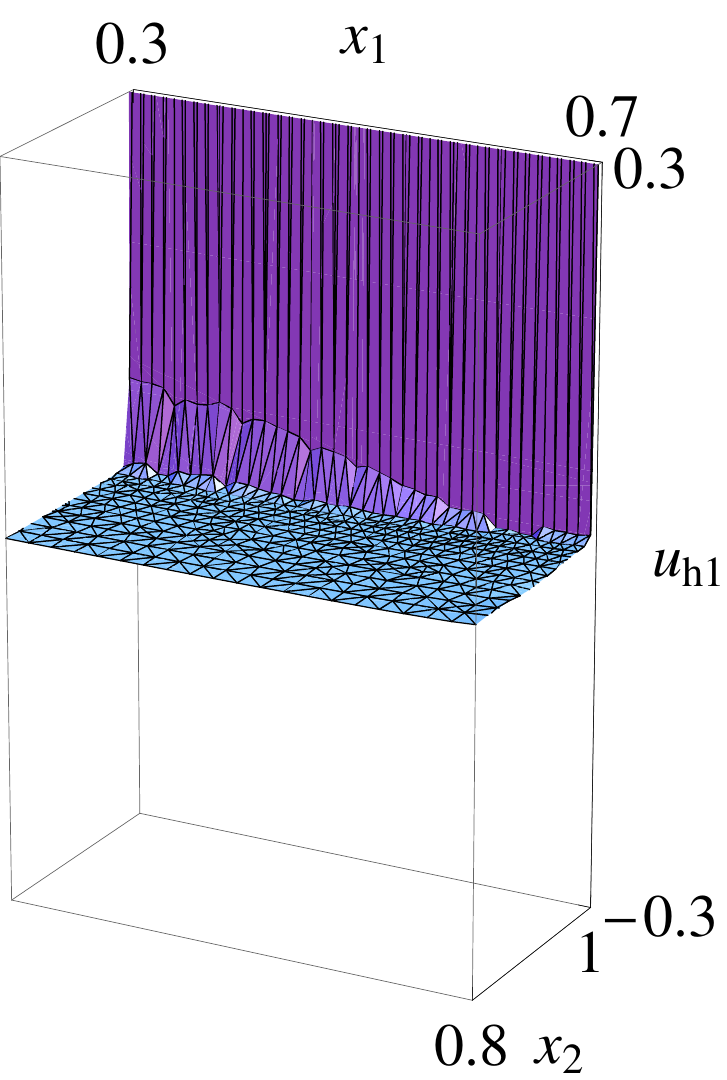} 
	\quad
	\includegraphics[width=34mm]{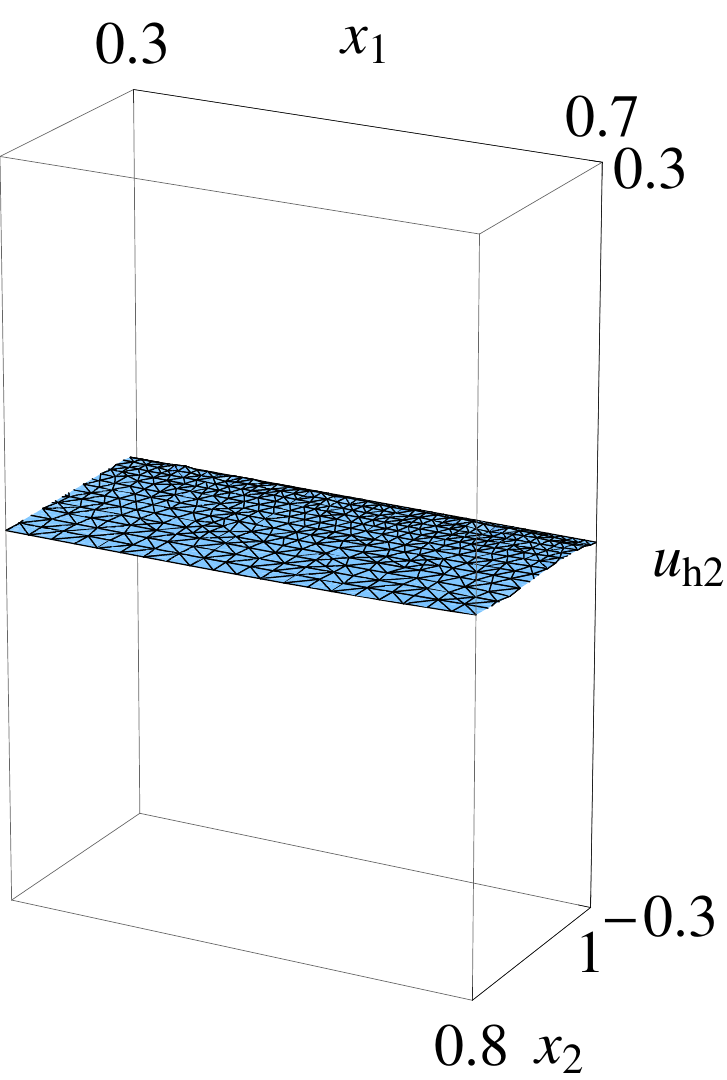} 
	\caption{
	Stereographs of $u_{h1}^n$ (left) and $u_{h2}^n$ (right) at $t^n=8$ by \zettai \ in Example \ref{exam:cav2} when $\nu=10^{-5}$ 
	}
	\label{fig:cavnu2-5sch2}
}

Figure \ref{fig:cavityBoundaryAndMesh} (right) shows the triangulation of $\Omega$.
Figures \ref{fig:cavnu-4sch1} and \ref{fig:cavnu-4sch2} show the stereographs of the solution $u_h^n$ at $t^n=8$ in the subdomain $(0.3,0.7)\times(0.8,1.0)$ by \jurai \ and \zettai, respectively, when $\nu=10^{-4}$. 
Neither solution is oscillating although $u_{h2}$ of \jurai \ takes larger values than that of \zettai.
Figures \ref{fig:cavnu2-5sch1} and \ref{fig:cavnu2-5sch2} show the stereographs of the solution $u_h^n$ at $t^n=8$ in the subdomain $(0.3,0.7)\times(0.8,1.0)$ by \jurai \ and \zettai, respectively, when $\nu=10^{-5}$.
While oscillation is observed for both components of the solution by {\jurai} in Figure \ref{fig:cavnu2-5sch1}, we can see that the solution by {\zettai} is solved without any oscillation in Figure \ref{fig:cavnu2-5sch2}.  
%

\section{Conclusions}\label{sec:conclusion}
We have present a Lagrange--Galerkin scheme free from numerical quadrature for the Navier--Stokes equations.
By virtue of the introduction of a locally linearized velocity, the scheme can be implemented exactly and the theoretical stability and the convergence results are assured for practical numerical solutions.  
We have shown optimal error estimates in $\ell^\infty(H^1)\times \ell^2(L^2)$-norm for the velocity and pressure in the case of $\pk 2/\pk 1$- and $\mini$-finite elements.
Numerical results have reflected these estimates and the robustness of the scheme for high-Reynolds number problems.  

 \begin{acknowledgements}
The first author was supported by JSPS (Japan Society for the Promotion of Science) under Grants-in-Aid for Scientific Research (C), No. 25400212 and (S), No. 24224004 and 
under the Japanese-German Graduate Externship (Mathematical Fluid Dynamics) and by Waseda University under Project research, Spectral analysis and its application to the stability theory of the Navier-Stokes equations of Research Institute for Science and Engineering.
The second author was supported by JSPS under Grant-in-Aid for JSPS Fellows, No. 26$\cdot$964.
\end{acknowledgements}

%
%

\bibliographystyle{spmpsci}      
\bibliography{bibArt2}   

%
%
%
\end{document}